\documentclass[reqno]{amsart}
\usepackage{amssymb}
\usepackage{enumitem}

\usepackage{hyperref}
\usepackage{verbatim}

\usepackage{esint}

\theoremstyle{plain}
\newtheorem{theorem}{Theorem}[section]
\newtheorem{lemma}[theorem]{Lemma}
\newtheorem{corollary}[theorem]{Corollary}
\newtheorem{proposition}[theorem]{Proposition}

\theoremstyle{definition}
\newtheorem{definition}[theorem]{Definition}

\newtheorem{assumption}[theorem]{Assumption}

\theoremstyle{remark}
\newtheorem{remark}[theorem]{Remark}

\newcommand{\supp}{\operatorname{supp}}

\numberwithin{equation}{section}

\newcommand{\bR}{\mathbb{R}}

\newcommand\mA{\mathcal{A}}
\newcommand\mB{\mathcal{B}}
\newcommand\mM{\mathcal{M}}

\providecommand{\norm}[1]{\lVert#1\rVert}

\makeatletter
\def\dashint{\operatorname%
{\,\,\text{\bf-}\kern-.98em\DOTSI\intop\ilimits@\!\!}}
\makeatother

\renewcommand{\vec}[1]{\boldsymbol{#1}}

\newcommand{\p}{\partial}
\newcommand{\epsi}{\varepsilon}
\newcommand{\vu}{\vec{u}}
\newcommand{\vv}{\vec{v}}
\newcommand{\vw}{\vec{w}}

\newcommand{\vf}{\vec{f}}
\newcommand{\vg}{\vec{g}}
\newcommand{\vvarphi}{\vec{\varphi}}

\makeatletter
\@namedef{subjclassname@2020}{%
  \textup{2020} Mathematics Subject Classification}
\makeatother

\begin{document}

\title[The conormal and Robin boundary value problems]{The conormal and Robin boundary value problems in nonsmooth domains satisfying a measure condition}

\author[Hongjie Dong]{Hongjie Dong}	
	
\address{Division of Applied Mathematics, Brown University, 182 George Street, Providence, RI 02912, USA}
\email{hongjie\_dong@brown.edu}

\author[Zongyuan Li]{Zongyuan Li}
	
\address{Division of Applied Mathematics, Brown University, 182 George Street, Providence, RI 02912, USA}
\email{zongyuan\_li@brown.edu}

\begin{abstract}
We consider elliptic equations and systems in divergence form with the conormal or the Robin boundary conditions, with small BMO (bounded mean oscillation) or variably partially small BMO coefficients. We propose a new class of domains which are locally close to a half space (or convex domains) with respect to the Lebesgue measure in the system (or scalar, respectively) case, and obtain the $W^1_p$ estimate for the conormal problem with the homogeneous boundary condition. Such condition is weaker than the Reifenberg flatness condition, for which the closeness is measured in terms of the Hausdorff distance, and the semi-convexity condition. For the conormal problem with inhomogeneous boundary conditions, we also assume that the domain is Lipschitz. By using these results, we obtain the $W^1_p$ and weighted $W^1_p$ estimates for the Robin problem in these domains.
\end{abstract}



\maketitle

\section{Introduction}
In this paper, we consider second-order elliptic equations and systems in divergence form
\begin{equation}\label{eqn-200622-1036}
	D_i(a_{ij}^{\alpha\beta}D_j u^{\beta})=D_i f_i^\alpha+g^\alpha\quad\text{in}\,\,\Omega,
\end{equation}
where $\Omega\subset \bR^d$ is a bounded or unbounded domain. Here the unknown is a vector $(u^\alpha)_{\alpha=1}^m$ and a typical point is $x=(x_1,\ldots,x_d)\in \bR^d$. We assume the strong ellipticity condition: for some $\kappa\in(0,1]$ and for any $x\in \bR^d$,
\begin{equation}\label{eqn-200701-0518}
	a_{ij}^{\alpha\beta}(x)\xi^\alpha_i\xi^\beta_j\geq \kappa |\xi|^2\quad\forall\ \xi=(\xi_i^\alpha)\in \bR^{d\times m}\quad\text{and}\,\, |a_{ij}^{\alpha\beta}(x)|\leq \kappa^{-1}.
\end{equation}
Let $n=(n_1,\ldots,n_d)$ be the outer normal direction on $\partial\Omega$.
We are interested in two types of boundary conditions: the conormal boundary condition
\begin{equation}\label{eqn-200609-0514}
	a_{ij}^{\alpha\beta}D_j u^\beta n_i = f_i^\alpha n_i + \varphi^\alpha\quad \text{on}\,\, \partial \Omega,
\end{equation}
and the Robin boundary condition
\begin{equation}\label{eqn-200527-1002}
a_{ij}^{\alpha\beta}D_j u^\beta n_i+ \gamma^{\alpha\beta} u^\beta=f_i^\alpha n_i\quad \text{on}\,\, \partial \Omega.
\end{equation}
For the Robin problem, we impose the following non-degeneracy condition on $\gamma^{\alpha\beta}=\gamma^{\alpha\beta}(x)$: for some $\gamma_0\in (0,1)$ and $E\subset \p\Omega$ with $|E|>0$,
\begin{equation}\label{eqn-200717-0438}
	\gamma^{\alpha\beta}\eta^\alpha\eta^\beta \geq \gamma_0|\eta|^2
\,\, \text{on}\,\,E\subset \p\Omega \quad\forall \ \eta\in \bR^d,
\quad |\gamma^{\alpha\beta}|\leq \gamma_0^{-1}.
\end{equation}
For the conormal problem, we also assume the necessary compatibility condition
$$
\int_{\partial\Omega}\varphi^\alpha=  \int_\Omega g^\alpha,
$$
which is not needed for the Robin problem.

The $L_p$ estimates for elliptic equations with the conormal or the Robin boundary condition have been studied extensively in the literature. See, for instance, \cite{BW05, MR3013054,YCYY20,YYY20} and the references therein. The novelty of our results below is that the domain $\Omega$ can be very irregular.

The conormal problem for second and higher order elliptic systems in Reifenberg flat domains was studied in \cite{MR3013054}. Roughly speaking, in all small scale the boundary of a Reifenberg flat domain is trapped in a thin disc. The leading coefficients are assumed to be of variably partially small BMO, i.e., they are merely measurable in one direction and have small mean oscillation in the orthogonal directions in each small ball, with the direction allowed to depend on the ball. We note that such class of coefficients was first introduced in \cite{MR2540989}. We also refer the reader to \cite{BW05} for an earlier result about scalar equations with small BMO coefficients, and \cite{G12} for elliptic equations and systems with symmetric small BMO coefficients in Lipschitz domain when $p$ is in restricted ranges, as well as \cite{YCYY20} for a weighted estimate for scalar elliptic equations with symmetric small BMO coefficients in bounded semi-convex domains.

This paper is motivated by recent work \cite{YYY20}, where among other results the authors obtained the weighted $W^1_p$ estimates for the Robin problem of scalar elliptic equations with symmetric and small BMO coefficients in bounded $C^1$ or semi-convex domains in $\bR^d,d\ge 3$. Recall that a domain is semi-convex if and only if it is Lipschitz and satisfies the uniform exterior ball condition. The proof in \cite{YYY20} relies on the Moser iteration as well as a reverse H\"older's inequality for $Du$. Especially, the authors used a reverse H\"older's inequality for the non-tangential maximal function of $Du$ on $\partial\Omega$ for a harmonic function $u$ satisfying the Robin boundary condition established in \cite[Theorems 1.2 and 3.2]{YYY18}. The proofs of the latter two theorems are based on the estimates of the fundamental solution of the Laplace operator and a solvability result obtained in \cite[Theorem 4.3]{MR2038145}. Another important ingredient of the proof is \cite[Theorems 3.1]{YYY20}, which is in the spirit of a level set argument originally used in \cite{MR1486629}. See also, for example, in \cite{BW05,S07,MR3013054}.

The objective of this paper is twofold. First, we generalize the results in \cite{BW05, MR3013054, YCYY20} for elliptic equations and systems with the conormal boundary conditions to more general domains. For elliptic systems with the homogeneous conormal boundary condition, we assume that in small scale $\partial\Omega$ is locally close to hyperplanes with respect to the Lebesgue measure in $\bR^d$. As to the coefficients, similar to \cite{MR3013054}, we impose the variably partially small BMO condition. See Assumption \ref{ass-200604-1252} for details.
Our condition on $\Omega$ can be even weaker for scalar equations with symmetric coefficients: in small scale it is locally close to convex domains with respect to the Lebesgue measure. However, in this case, we additionally impose the small BMO condition on the leading coefficients. See the $W^1_p$ estimates and solvability in Theorems \ref{thm-200609-0520} and \ref{thm-200617-1154}. Such conditions are weaker than the Reifenberg flat condition used in \cite{BW05, MR3013054}, where the flatness of $\partial\Omega$ is measured in terms of the Hausdorff distance. For the conormal problem with inhomogeneous boundary conditions, we assume that the domain is also Lipschitz so that the trace and extension operators are well defined. See Theorems \ref{thm-200609-0525} and \ref{thm-200618-1200}. This is also weaker than the $C^1$ and the semi-convexity conditions used, for example, in \cite{Y16,YCYY20}. In the proof, to deal with the nonsmooth boundary $\partial\Omega$ we adapted a reflection argument used in \cite{MR3013054} with a new observation that to estimate certain error terms arising from the reflection, it suffices to have a small measure condition instead of the flatness condition. For scalar equations, we also exploit a local Lipschitz estimate in \cite{MR3168044} for harmonic functions satisfying the Neumann boundary condition in convex domains.

Second, we extend the result in \cite{YYY20} for the Robin problem to elliptic equations and systems with small BMO or variably partially small BMO coefficients in domains satisfying the same conditions as in the conormal case. See Theorems \ref{thm-200527-1035} and \ref{thm-200622-1138}. Similar to \cite{YYY20}, a key ingredient of our proof is a reverse H\"older's inequality for the Robin problem with homogeneous right-hand side. See Lemma \ref{lemma1}. However, our proof of such reverse H\"older's inequality is very different from that in \cite{YYY20}. In particular, we do not use the Moser iteration which is not valid for elliptic systems, or any non-tangential maximal function estimates.
Instead, we appeal to our results for the conormal problem (Theorems \ref{thm-200609-0525} and \ref{thm-200618-1200}) and use a delicate decomposition and localization argument. Given the corresponding results for the Dirichlet problem (see, for instance, \cite{DK11b}),
it seems to us that for the conormal or Robin problem the domain under consideration can be less regular.

For simplicity, in this paper we choose not to consider lower-order terms. Our results can be readily extended to elliptic systems of the form
$$\left\{
\begin{aligned}
D_i(a_{ij}^{\alpha\beta}D_j u^{\beta}+b_i^{\alpha\beta}u^\beta)+\tilde b_i^{\alpha\beta}D_iu^\beta+c^{\alpha\beta}u^\beta-\lambda u^\alpha&=D_i f_i^\alpha+g^\alpha\quad\text{in}\,\,\Omega,\\
(a_{ij}^{\alpha\beta}D_j u^\beta+b_i^{\alpha\beta}u^\beta)n_i &= f_i^\alpha n_i + \varphi^\alpha\quad\text{on}\,\,\p\Omega,
\end{aligned}\right.
$$
when $b_i^{\alpha\beta}$, $\tilde b_i^{\alpha\beta}$, and $c^{\alpha\beta}$ are bounded and $\lambda>0$ is sufficiently large constant. In the case of scalar equations, we can take $\lambda=0$ under proper sign conditions. 
See \cite[Theorems 2.4 and 2.6]{MR3013054} for details. Similar extensions can be made to the results about the Robin problem.

The remaining part of paper is organized as follows. In the next section, we state two groups of the main results: first for elliptic systems and then for scalar equations. In Section \ref{sec3}, we prove the estimates and solvability for the conormal problem of elliptic systems. Section \ref{sec4} is devoted to the proof of Theorem \ref{thm-200527-1035}, which is regarding the estimates and solvability for the Robin problem of elliptic systems. In Section \ref{sec5}, we deal with scalar elliptic equations with symmetric coefficients in more general domains.

\noindent{\bf Notation.}
Suppose $\Omega\subset\bR^d$ is a domain and $p\in (1,\infty)$. For a point $x=(x_1,\ldots,x_d) = (x',x_d)$, we define
$$
\Omega_r(x)=\Omega\cap B_r(x),\quad B_r'(x') = \{y':|y'-x'|<r\}.
$$
We use the following notation for the average in terms of the Lebesgue measure: for a set $B$ with finite Lebesgue measure $|B|$, let
\begin{equation*}
	(f)_{B} = \fint_\Omega f = \frac{1}{|B|}\int_B f.
\end{equation*}

Define the space $L^1_p$ to be the collection of all functions $u\in L_{1,\text{loc}}(\Omega)$ with $Du\in L_p(\Omega)$, equipped with the semi-norm $\norm{Du}_{L_p(\Omega)}$. We also consider its quotient space $\dot{L}^1_p=L^1_p/c$, for which we identify two functions that only differ by a constant. Then $\dot{L}^1_p$ equipped with $\norm{Du}_{L_p(\Omega)}$ is a Banach space. For details, see \cite[Section~1.1]{MR2777530}.
It is worth mentioning that, when the Poincar\'e inequality
$$
\norm{u-(u)_\Omega}_{L_p(\Omega)} \leq N \norm{Du}_{L_p(\Omega)}
$$
holds, the space $\dot{L}^1_p$ can be identified with the usual Sobolev space
$$W^1_p(\Omega)\cap \{(u)_\Omega =0\}.$$

The weighted Sobolev spaces will also be discussed in this paper. A weight $\omega$ is a non-negative, locally integrable function. For $p\in(1,\infty)$, $\omega$ is said to be of Muckenhoupt $A_p$ class if
\begin{equation*}
	[\omega]_{A_p} := \sup_{B\subset \bR^d} \left(\fint_B \omega\right)\left(\fint_B \omega^{-1/(p-1)}\right)^{p-1}<\infty,
\end{equation*}
where the sup is taken over all balls. For any measurable set $E$ and a measurable function $u$ defined on $\Omega$, we denote
\begin{equation*}
	\omega(E):= \int_E \omega(x)\,dx,\quad \norm{u}_{L_{p,\omega}} := \left(\int_\Omega u(x)^p\omega(x)\,dx\right)^{1/p}.
\end{equation*}
The weighted Sobolev space $W^1_{p,\omega}$ is defined similarly. The weighted spaces with the weight $\omega=1$ are the same with the usual $L_p$ and $W^1_p$ spaces.

For a number $p\in(1,\infty)$, we denote
\begin{equation*}
	p^*:=\begin{cases}
		\frac{pd}{p+d}\quad &p\in (\frac{d}{d-1},\infty),\\
		1+\epsi\quad &p = \frac{d}{d-1},\\
		1 \quad &p \in (1, \frac{d}{d-1}),
	\end{cases}
\end{equation*}
where $\epsi$ can be any fixed positive number.

For elliptic systems, we sometimes use the following notation for vector-valued functions
\begin{equation*}
	\vu:=(u^\alpha)_{\alpha=1}^m,\quad\vf_i:=(f_i^\alpha)_{\alpha=1}^m,
\quad \vg:=(g^\alpha)_{\alpha=1}^m,
\quad\text{and}\quad \vec{\varphi}:=(\varphi^\alpha)_{\alpha=1}^m.
\end{equation*}


\section{Main results}
\subsection{Results for elliptic systems}
For elliptic systems, we consider domains that are $\theta$-close to hyperplanes (Assumption \ref{ass-200604-1252} (b) below) and variably partially BMO coefficients--those coefficients that are measurable in the ``almost normal'' direction while BMO in the other variables.

In Assumption \ref{ass-200604-1252} as well as Assumptions \ref{ass-200617-1149} and \ref{ass-200519-1052} below, $R_0$ and $M$ are fixed constants, and $\theta$ is a small parameter to be determined later.
\begin{assumption}[$\theta$]\label{ass-200604-1252}
	There exists a constant $R_0\in(0,1]$ such that the following hold.
	\begin{enumerate}[wide, labelwidth=0pt, labelindent=0pt]
		\item For any $x\in \Omega$ and any $r\in(0,R_0]$ such that $B_r(x)\subset\Omega$, there is an orthogonal coordinate system $y=(y',y_d)$ depending on $x$ and $r$ such that in the new coordinate system $x=0$ and
		\begin{equation}\label{eqn-200604-1228}
		\fint_{B_r}\left|a_{ij}^{\alpha\beta}(y',y_d)-\fint_{B_r'}a_{ij}^{\alpha\beta}(z',y_d)\,dz'\right|\,dy \leq \theta.
		\end{equation}
		\item For any $x\in\p\Omega$ and $r\in(0,R_0]$, besides \eqref{eqn-200604-1228} we also require that in the new coordinate system,
		\begin{align*}
			&\Omega^+_r:=\{y:\ y'\in B_r',\ 0<y_d<Mr\}\subset \Omega\cap (B'_{r}\times (-Mr,Mr)),\notag\\
			&\Omega\cap ((-\infty,-Mr)\times B'_{r})=\emptyset,\notag\\
			&|\Omega^-_r|\le \theta |B_r|,\quad\text{where}\,\, \Omega^-_r:=\{y'\in B_r',\ -Mr<y_d<0,\ y\in \Omega\},
		\end{align*}
where $M\geq 1$ is a constant which can be chosen to be independent of $x$.
	\end{enumerate}
\end{assumption}
Assumption \ref{ass-200604-1252} (b) is a generalization of the usual small Reifenberg flat condition. Here instead of the closeness in terms of the Hausdorff distance, the closeness is measured in terms of the Lebesgue measure. Note that the Reifenberg flat condition (see, for instance,  \cite{DK11b,MR3013054}) implies Assumption \ref{ass-200604-1252} (b).

\begin{definition}
We say that a function $\vu\in \dot L^1_p(\Omega),p\in (1,\infty),$ is a weak solution to the problem \eqref{eqn-200622-1036}-\eqref{eqn-200609-0514} if for any $\vec{\phi}\in C^\infty(\overline{\Omega})$,
\begin{equation}\label{eqn-200703-0330}
	-\int_\Omega a_{ij}^{\alpha\beta}D_ju^\beta D_i\phi^\alpha = -\int_\Omega f_i^\alpha D_i\phi^\alpha + \int_\Omega g^\alpha\phi^\alpha - \int_{\p\Omega}\varphi^\alpha\phi^\alpha.
\end{equation}
\end{definition}

Our first result is an a priori estimate for the problem \eqref{eqn-200622-1036}-\eqref{eqn-200609-0514} with zero non-divergence source term $\vg=0$ and homogeneous boundary condition $\vec{\varphi}=0$.
\begin{theorem}\label{thm-200609-0520}
Suppose that $\Omega$ is a domain in $\bR^d$ (not necessarily bounded), $p\in (1,\infty)$, and $\vu\in \dot{L}^1_p(\Omega;\bR^m)$ is a weak solution to \eqref{eqn-200622-1036}-\eqref{eqn-200609-0514} with $\vf_i\in L_p(\Omega)$ and $\vg= \vvarphi= 0$. Then we can find a sufficiently small $\theta_0$ depending on $(d,m,\kappa,M,p)$, such that if Assumption \ref{ass-200604-1252} ($\theta_0$) is satisfied, then we have
	\begin{equation}\label{eqn-200619-1150}
	\norm{D\vu}_{L_{p}(\Omega)}\leq N\left(R_0^{d(1/p-1)}\norm{D\vu}_{L_1(\Omega)} + \norm{\vf_i}_{L_{p}(\Omega)}\right),
	\end{equation}
where $N=N(d,m,\kappa,M,p)$.
\end{theorem}

We give the following remark on the solvability of \eqref{eqn-200622-1036}-\eqref{eqn-200609-0514} with $\vg=\vvarphi=0$.
\begin{remark}\label{rmk-200619-1150}
	\begin{enumerate}[wide, labelwidth=0pt, labelindent=0pt]
		\item The $\dot{L}^1_2$ unique solvability follows from the $L_2$ estimate by testing the equation by $u$ and the Lax-Milgram lemma.
		\item Under the assumptions of Theorem \ref{thm-200609-0520}, if we further assume that $|\Omega|<\infty$ and the local Sobolev-Poincar\'e inequality with index $(2,2-\epsi)$ holds for some $\epsi>0$, i.e., there exists some constant $N_{\Omega,\epsi}$, such that for any $v\in L_2\cap\dot{L}^1_{2-\epsi}$, all small scales $r<R_0$, and $x\in \overline{\Omega}$,
		\begin{equation}\label{eqn-200619-0355}
		\left(\fint_{\Omega_r(x)}\Big|v-\fint_{\Omega_r(x)}v\Big|^2\right)^{1/2} \leq N_{\Omega,\epsi} r \left(\fint_{\Omega_{2r}(x)}|Dv|^{2-\epsi}\right)^{1/(2-\epsi)},
		\end{equation}
then we have the unique $\dot{L}^1_p$ solvability for any $p\in (1,\infty)$. Furthermore, the lower-order term $\norm{D\vu}_{L_{1}(\Omega)}$ on the right-hand side of \eqref{eqn-200619-1150} can be dropped. Indeed, from the local energy estimate (the Caccioppoli inequality), \eqref{eqn-200619-0355}, and Gehring's lemma, we get a local higher integrability of $D\vu$ for $\dot{L}^1_2$-solution, which together with the proof of Theorem \ref{thm-200609-0520} below shows that, for any $p\in (2,\infty)$, the $\dot{L}^1_2$-solution in (a) is also in $\dot{L}^1_p$. The solvability when $p\in (1,2)$ then follows from a duality argument. For more details, see Section \ref{sec-200705-1036}.
	\end{enumerate}

The so-called Jones flat domains in \cite{MR631089} are local $W^1_p$ extension domains for any $p\in[1,\infty)$, which further implies \eqref{eqn-200619-0355}. In this case, we can take the index as in the Sobolev-Poincar\'e inequality: $2-\epsi = 2^*$. For a more concrete example, see \cite[Theorem~3.3]{MR3809039} for the Reifenberg flat domain.
\end{remark}

Now we turn to a discussion on bounded Lipschitz domains.
\begin{definition}
A domain $\Omega\subset \bR^d$ is called a Lipschitz domain (with Lipschitz constant $M$), if there exits a constant $R_0>0$, such that for and any $x\in\p\Omega$, we can find a Lipschitz function $\psi:\bR^{d-1}\rightarrow \bR$ with $|D\psi|\leq M$ and a coordinate system $y=(y',y_d)$, in which
	\begin{equation*}
		\Omega_{R_0}(x) = \{y: y_d>\psi(y')\}.
	\end{equation*}
\end{definition}
Note that any Lipschitz domain is an $W^1_p$-extension domain for $p\in[1,\infty]$ and the trace operator $W^1_p(\Omega)\rightarrow W^{1-1/p}_p(\p\Omega)$ is well defined. In this case, we consider the solvability in $W^1_p(\Omega)$ of \eqref{eqn-200622-1036}-\eqref{eqn-200609-0514} with general nonzero $\vg \in L_{p^*}(\Omega)$ and $\vvarphi\in W^{-1/p}_p(\partial\Omega)$ satisfying the compatibility condition \eqref{eqn-200717-0438}. Here $W^{-1/p}_p(\partial\Omega)$ denotes the dual space of $W^{1/p}_{p'}(\partial\Omega)$, where $p'=p/(p-1)$.
\begin{theorem}\label{thm-200609-0525}
	Suppose that $\Omega$ is a bounded Lipschitz domain and $p\in(1,\infty)$. Then we can find a sufficiently small $\theta_0>0$ depending on $(d,m,\kappa,M,p)$, such that if Assumption \ref{ass-200604-1252} ($\theta_0$) is satisfied, the $W^1_p$ well-posedness holds for the problem \eqref{eqn-200622-1036}-\eqref{eqn-200609-0514}. Namely, for any $\vf_i\in L_p(\Omega)$, $\vg\in L_{p^*}(\Omega)$, and $\vvarphi\in W^{-1/p}_p(\p\Omega)$ satisfying \eqref{eqn-200717-0438}, there exists a unique solution $\vu\in W^1_p(\Omega;\bR^m)$ with mean zero, satisfying
	\begin{equation*}
	\norm{\vu}_{W^1_{p}(\Omega)} \leq N\sum_i \norm{\vf_i}_{L_{p}(\Omega)} + N\norm{\vg}_{L_{p^*}(\Omega)} + N\norm{\vvarphi}_{W^{-1/p}_p(\p\Omega)},
	\end{equation*}
where $N=N(d,m,\kappa,M,p,R_0,|\Omega|)$.
\end{theorem}

The weak formulation of the Robin problem \eqref{eqn-200622-1036}-\eqref{eqn-200527-1002} is given as follows.
\begin{definition}
We say that a function $\vu\in \dot W^1_p(\Omega),p\in (1,\infty),$ is a weak solution to the problem \eqref{eqn-200622-1036}-\eqref{eqn-200527-1002} if for any $\vec{\phi}\in C^\infty(\overline{\Omega})$,
\begin{equation*}
	-\int_\Omega a_{ij}^{\alpha\beta}D_ju^\beta D_i\phi^\alpha - \int_{\p\Omega}\gamma^{\alpha\beta}u^\beta\phi^\alpha = -\int_\Omega f_i^\alpha D_i\phi^\alpha + \int_\Omega g^\alpha\phi^\alpha .
\end{equation*}
\end{definition}

For the Robin problem, we prove the following result.
\begin{theorem}\label{thm-200527-1035}
Suppose that $\Omega$ is a bounded Lipschitz domain, $p\in(1,\infty)$, and $\omega\in A_p$. We can find a sufficiently small $\theta_0$ depending on $(d,m,\kappa,M,p,[\omega]_{A_p})$, such that if Assumption \ref{ass-200604-1252} ($\theta_0$) is satisfied, then we have the $W^1_{p,\omega}$ well-posedness for \eqref{eqn-200622-1036}-\eqref{eqn-200527-1002}: for any $\vf_i\in L_{p,\omega}$ and $\vg\in L_{q,\omega^a}$, there exists a unique solution $\vu\in W^1_{p,\omega}(\Omega;\bR^m)$ satisfying
	\begin{equation*}
		\norm{\vu}_{W^1_{p,\omega}(\Omega)}\leq N\sum_i\norm{\vf_i}_{L_{p,\omega}(\Omega)} + N \norm{\vg}_{L_{q,\omega^a}(\Omega)},
	\end{equation*}
where the constant $N=N(d,m,\kappa,M,p,R_0,\gamma_0,\Omega,E,[\omega]_{A_p})$. The exponents $q$ and $a$ can be taken from either of the following two cases.
	\begin{enumerate}[wide, labelwidth=0pt, labelindent=0pt]
		\item $q= p_\omega := \frac{1}{1/p + 1/(\hat{p})^* - 1/\hat{p}}(<p)$ and $a = p_\omega/p$, where $\hat{p}>1$ depends on $[\omega]_{A_p}$ and can be any constant satisfying Lemma \ref{lem-200615-1006}(b), which is usually close to $1$ for general $\omega\in A_p$.
		\item $q=\frac{dp}{d+p-1}$ and $a=\frac{d-1}{d+p-1}$.
	\end{enumerate}
\end{theorem}
\begin{remark}
	\label{rem2.6}
	Let us make some remarks on the index $p_{\omega}$. By H\"older's inequality, we have $L_{p,\omega}(\Omega)\subset L_{p_\omega,\omega^{p_\omega/p}}(\Omega)$. Hence the theorem holds for $\vg\in L_{p,\omega}(\Omega)$. Also, when $\omega\equiv 1$, we can take $\hat{p}=p$. In this case, $p_\omega=p^*$, which recovers the unweighted estimate in Theorem \ref{thm-200614-1046}. Note that we always have $p_\omega\ge p^*$.
\end{remark}
\begin{remark}
		The boundedness of $\gamma$ in Theorem \ref{thm-200527-1035} can be generalized. If the weight $\omega$ is in the reverse H\"older class $RH_s$ (see the definition in Lemma \ref{lem-200615-1006}), then for the \textbf{case (a)}, it suffices to assume
	\begin{equation}\label{eqn-200709-1214}
	\gamma\in \begin{cases}
		L_{\frac{ps(d-1)}{(s-1)d}+\epsi}\quad &\text{when}\,\,ps/(s-1)\geq d\\
		L_{d-1 +\epsi} &\text{when}\,\,ps/(s-1)< d,
	\end{cases}
	\end{equation}
	where $\epsi$ can be an arbitrarily small positive number. For the \textbf{case (b)}, we also require conditions coming from the duality, say, \eqref{eqn-200709-1214} with $(p,s)$ replaced with $(p',s')$ where $1/p+1/p'=1$ and $s'$ is the constant such that $\omega^{-p'/p}\in RH_{s'}$. Note that $s$ and $s'$ are typically slightly bigger than $1$.
	
When $\omega\equiv 1$, we can take $s=s'=\infty$. Except for an $\epsi$ loss in the case when $p>d$, the condition \eqref{eqn-200709-1214} is consistent with the one in Remark \ref{rmk-200709-0752} for the unweighted solvability and estimate.
\end{remark}

\begin{remark}
In the above theorems, when the leading coefficients are of small BMO (in all variables, instead of partial BMO), the strong ellipticity condition can be generalized to the Legendre-Hadamard condition, i.e., the condition \eqref{eqn-200701-0518} holds with $\xi_i^\alpha$ being of the form $\zeta_i\eta^\alpha$.
\end{remark}

\subsection{Results for scalar elliptic equations}
The last part of this paper will be devoted to the scalar equation
\begin{equation}
\label{eq11.20}
Lu := D_i(a_{ij}D_j u)=D_i f_i+g\quad\text{in}\quad \Omega.
\end{equation}
Clearly, the scalar equation is a special case of the aforementioned elliptic system. We prove that the assumption on the $\p\Omega$ can be relaxed, when the leading coefficients are symmetric and of small BMO. Let us assume that the leading coefficients satisfy
\begin{equation*}
	a_{ij}\xi_i\xi_j\geq \kappa |\xi|^2\quad \forall \xi\in \bR^d,\quad |a_{ij}|\leq \kappa^{-1},\quad a_{ij}=a_{ji}.
\end{equation*}
\begin{assumption}[$\theta$]\label{ass-200617-1149}
	There exists some $R_0>0$, such that $$\sup_{x, 0<r<R_0}\fint_{B_r(x)\cap \Omega}|a_{ij}(y)-(a_{ij})_{B_r(x)\cap\Omega}|\,dy < \theta.$$
\end{assumption}
In the following assumption, instead of a $(d-1)$-dimensional hyperplane, $\partial\Omega$ is locally close (in the sense of the Lebesgue measure) to the boundary of a convex domain.
\begin{assumption}[$\theta$]\label{ass-200519-1052}
For any $x\in\p\Omega$ and $r\le R_0$, there is an orthogonal coordinate system in which $x=0$ and
	\begin{align*}
		&\Omega^+_r:=\{x:\ x'\in B_r',\ \psi(x')<x_d<2Mr\}\subset \Omega\cap ((-Mr,2Mr)\times B'_{r}),\notag\\
		&\Omega\cap ((-\infty,-Mr)\times B'_{r})=\emptyset,\notag\\
		&|\Omega^-_r|\le \theta |B_r|,\quad\text{where}\,\, \Omega^-_r=\{x: \ x'\in B_r',\ -Mr<x_d<\psi(x'),\ x\in \Omega\},
	\end{align*}
	where $\psi$ is a convex function (with Lipschitz constant say $M/10$ in the chosen coordinate system).
\end{assumption}
It is easy to check that the semi-convex domains considered in \cite{Y16,YCYY20} and \cite{YYY20} satisfy Assumption \ref{ass-200519-1052} with vanishing constant, i.e., for arbitrarily small $\theta>0$, we can find $R_0(\theta)>0$ such that Assumption \ref{ass-200519-1052} ($\theta$) is satisfied. Note that in the theorems below, the choice of small $\theta$ is independent of the constant $R_0$, which means that our theorems hold for semi-convex domains.

As in the system case, we first consider the conormal boundary condition
\begin{equation}\label{eqn-200617-1155}
a_{ij}D_j u n_i = f_i n_i + \varphi\quad \text{on}\quad \partial \Omega.
\end{equation}
In the same spirit as Theorems \ref{thm-200609-0520} and  \ref{thm-200609-0525}, we prove the following two theorems regarding the conormal boundary value problem.
\begin{theorem}\label{thm-200617-1154}
	Suppose that $\Omega$ is a domain in $\bR^d$ (not necessarily bounded), $p\in (1,\infty)$, and $u\in \dot{L}^1_p(\Omega)$ is a weak solution to \eqref{eq11.20}-\eqref{eqn-200617-1155} with $f_i\in L_p(\Omega)$ and $g = \varphi = 0$. Then we can find a sufficiently small $\theta_0$ depending on $(d,\kappa,M,p)$, such that if Assumption \ref{ass-200617-1149} ($\theta_0$) and \ref{ass-200519-1052} ($\theta_0$) are satisfied, then we have
	\begin{equation*}
		\norm{Du}_{L_{p}(\Omega)}\le N(R_0^{d(1/p-1)}\norm{Du}_{L_1(\Omega)} + \norm{f_i}_{L_{p}(\Omega)}),
	\end{equation*}
where $N=N(d,\kappa,M,p)$.
\end{theorem}

The discussions on the solvability is the same as in Remark \ref{rmk-200619-1150}, which we will omit here. As before, when $\Omega$ is also a Lipschitz domain, we can include non-zero $\varphi$ and $g$, with the compatibility
\begin{equation}\label{eqn-200621-0643}
\int_{\p\Omega}\varphi = \int_\Omega g.
\end{equation}
\begin{theorem}\label{thm-200618-1200}
	Consider the problem \eqref{eq11.20}-\eqref{eqn-200617-1155} on a bounded Lipschitz domain $\Omega$. For any $p\in(1,\infty)$, we can find a sufficiently small $\theta_0$ depending on $(d,\kappa,M,p)$, such that if Assumption \ref{ass-200617-1149} ($\theta_0$) and \ref{ass-200519-1052} ($\theta_0$) are satisfied, the $W^1_p$ well-posedness holds. Namely, for any $f_i\in L_p(\Omega)$, $g\in L_{p^*}(\Omega)$, and $\varphi\in W^{-1/p}_p(\p\Omega)$ satisfying \eqref{eqn-200621-0643}, there exists a unique solution $u\in W^1_p(\Omega)$ with mean zero, satisfying
	\begin{equation*}
	\norm{u}_{W^1_{p}(\Omega)}\leq N(\norm{f_i}_{L_{p}(\Omega)} + \norm{g}_{L_{p^*}(\Omega)} + \norm{\varphi}_{W^{-1/p}_p(\p\Omega)}),
	\end{equation*}
where $N=N(d,\kappa,M,p,R_0,|\Omega|)$ is a constant.
\end{theorem}
We also consider the Robin condition
\begin{equation}\label{eqn-200528-1239}
a_{ij}D_j u n_i+ \gamma u=f_i n_i\quad \text{on}\quad \partial \Omega.
\end{equation}
For the boundary operator, we assume for some $\gamma_0\in (0,1]$,
\begin{equation*}
	\gamma\geq \gamma_0\quad \text{on}\,\,E\subset\p\Omega\,\,\text{with}\,\,|E|>0,\quad \gamma \leq \gamma_0^{-1}.
\end{equation*}
\begin{theorem}\label{thm-200622-1138}
	On a bounded Lipschitz domain, we consider the problem \eqref{eq11.20}-\eqref{eqn-200528-1239}. For any $p\in(1,\infty)$, we can find a sufficiently small $\theta_0$ depending on $(d,\kappa,M,p,[\omega]_{A_p})$, such that if Assumption \ref{ass-200617-1149} ($\theta_0$) and \ref{ass-200519-1052} ($\theta_0$) are satisfied, then we have the $W^1_{p,\omega}$ well-posedness for any $A_p$ weight $\omega$. To be more specific, for any $f\in L_{p,\omega}, g\in L_{q,\omega^a}$, there exists a unique solution $u\in W^1_{p,\omega}$ satisfying
	\begin{equation*}
		\norm{u}_{W^1_{p,\omega}(\Omega)}\leq N( \norm{f_i}_{L_{p,\omega}(\Omega)} + \norm{g}_{L_{q,\omega^a}(\Omega)}),
	\end{equation*}
where the constant $N=N(d,\kappa,M,p,R_0,\gamma_0,\Omega,E,[\omega]_{A_p})$, and $q$ and $a$ can be taken as in Theorem \ref{thm-200527-1035}.
\end{theorem}
\section{Elliptic systems with conormal boundary conditions}
                            \label{sec3}

\subsection{Decomposition lemma}
The key step in proving Theorem \ref{thm-200609-0520} is the following decomposition lemma.
\begin{proposition}\label{prop-200611-0515}
Suppose that Assumption \ref{ass-200604-1252} $(\theta)$ is satisfied with some $\theta\in (0,1)$ and for some $q\in (1,\infty)$, $\vu\in \dot{L}^1_q(\Omega;\bR^m)$ is a weak solution to \eqref{eqn-200622-1036}-\eqref{eqn-200609-0514} with $\vf_i\in L_q(\Omega)$ and $\vg= \vvarphi = 0$. Then, for any $\tilde{p}\in (1,q)$, $r\in (0,R_0)$, and $x\in\overline{\Omega}$, we can find non-negative functions $V, W\in L_{\tilde{p}}(\Omega_{r/(4\sqrt{d}M)}(x))$ satisfying $|D\vu|\leq V + W$. Furthermore, the following estimates hold:
	\begin{equation}\label{eqn-200610-0250}
	(W^{\tilde{p}})^{1/\tilde{p}}_{\Omega_{r/(4\sqrt{d}M)}(x)} \le N \theta^{1/\tilde{p} - 1/q}(|D\vu|^{q})^{1/q}_{\Omega_r(x)} + N (F^{\tilde{p}})^{1/\tilde{p}}_{\Omega_{r}(x)},
	\end{equation}
	\begin{equation}\label{eqn-200611-0449}
	\norm{V}_{L_\infty(\Omega_{r/(8\sqrt{d}M)}(x))} \le N (|D\vu|^{q})^{1/q}_{\Omega_r(x)} + N(F^{\tilde{p}})^{1/\tilde{p}}_{\Omega_{r}(x)},
	\end{equation}
where $F=\sum_{i,\alpha}|f_i^\alpha|$ and $N=N(d,m,\kappa,M,q,\tilde p)>0$.
\end{proposition}

\begin{proof}
The proof is similar to that of \cite[Lemma~5.1]{MR3013054} even though our assumption on the boundary is more general. Below we only give a sketch of the proof. Let $c=1/(\sqrt{d}M)\le 1$. We discuss three cases.

\textbf{Case 1}: the interior case $B_{cr/4}(x)\subset \Omega$. The construction is the standard freezing coefficient procedure, in which no boundary condition is involved. See for example \cite[Lemma~8.3 (i)]{DK11b}.

\textbf{Case 2}: $x\in\p\Omega$. We will focus on this case. Here the smallness also comes from ``flattening'' the boundary. We choose the coordinate system as in Assumption \ref{ass-200604-1252}.
	
	We first construct $W$ using the reflection technique in \cite{MR3013054}. In the following, we denote
	\begin{equation*}
\Gamma_{cr}:=\p\Omega_{cr}^{+}\cap\{y^d=0\}.
	\end{equation*}
A direct computation shows that $u$ satisfies
	\begin{equation*}
		\begin{cases}
			D_i(\bar{a}_{ij}^{\alpha\beta}D_j u^{\beta}) = D_i\bar{f}_i^\alpha \quad&\text{in}\quad \Omega_{cr}^{+},\\
			\bar{a}_{ij}^{\alpha\beta}D_j u^\beta n_i = \bar{f}_i^\alpha n_i \quad &\text{on}\quad \Gamma_{cr},
		\end{cases}
	\end{equation*}
where,
	\begin{equation*}
		\bar{a}_{ij}^{\alpha\beta}=\bar{a}_{ij}^{\alpha\beta}(y_d)=\fint_{B_r'(0)}a_{ij}^{\alpha\beta}(y',y_d)\,dy'
	\end{equation*}
	and the source term $\{\bar{f}_i^\alpha\}$ is defined as follows. For $i=1,\ldots,d-1$,
	\begin{equation*}
		\bar f_i^\alpha=f_i^\alpha + D_i((\bar{a}_{ij}^{\alpha\beta} - a_{ij}^{\alpha\beta})D_ju^\beta) + 1_{T(y)\in \Omega^-_{cr}}
		\big(f_i^\alpha-a_{ij}^{\alpha\beta}D_j u^{\beta} \big)\circ T(y)
	\end{equation*}
and
		\begin{equation*}
		\bar f_d^{\alpha}=f_d^\alpha + D_d((\bar{a}_{dj}^{\alpha\beta} - a_{dj}^{\alpha\beta})D_ju^\beta)- 1_{T(y)\in \Omega^-_{cr}}\big(f_d^\alpha-a_{dj}^{\alpha\beta}D_ju^\beta\big)\circ T(y),
		\end{equation*}
where
\begin{equation*}
	T(y',y_d) = (y', -y_d).
\end{equation*}
	To see this, we can take any test function $\phi\in C^\infty_c(\Omega_{cr}^+\cup\Gamma_{cr})$ and construct a legitimate test function for the original problem on $\Omega\cap (B_{cr}'\times(-Mcr,Mcr))$ with conormal boundary condition on $\p\Omega\cap B_{cr}'\times(-Mcr,Mcr)$ by
	\begin{equation*}
		\mathcal{E}\phi:=\begin{cases}
			\phi\quad &\text{in}\,\,\Omega_{cr}^{+},\\
			\phi\circ T\quad &\text{in}\,\,\Omega_{cr}^{-}.
		\end{cases}
	\end{equation*}
	For details of this computation, we refer the reader to \cite[Lemma~4.2]{MR3013054}.
	
	Now we find a function $w\in W^1_{\tilde{p}}(\Omega_{cr}^{+})$ satisfying
	\begin{equation}\label{eqn-200611-1218}
	\begin{cases}
	D_i(\bar{a}_{ij}^{\alpha\beta}D_j w^{\beta}) = D_i\bar{f}_i^\alpha \quad&\text{in}\quad \Omega_{cr}^{+},\\
	\bar{a}_{ij}^{\alpha\beta}D_j w^\beta n_i = \bar{f}_i^\alpha n_i  \quad &\text{on}\quad \Gamma_{cr},
	\end{cases}
	\end{equation}
with $D\vw$ being controlled. Indeed, this can be achieved by considering an extended and mollified problem. Let $\tilde{B}$ be a smooth domain which is symmetric with respect to the $y_d$-hyperplane, with
	\begin{equation*}
		B_{cr}'\times(-Mcr,Mcr)\subsetneq\tilde{B} \subsetneq B_{2cr}'\times(-2Mcr,2Mcr).
	\end{equation*}
Next, we construct an extended problem on $\tilde{B}$ with coefficients $\tilde{a}_{ij}^{\alpha\beta}$. First, let $\tilde{a}_{ij}^{\alpha\beta} = \bar{a}_{ij}^{\alpha\beta}$ on $\Omega_{cr}^+$, and in $B'_{cr}\times (-Mcr,0)$,
	\begin{equation}\label{eqn-200702-0626}
		\tilde{a}_{ij}^{\alpha\beta}(y', y_d) := \begin{cases}
			\bar{a}_{ij}^{\alpha\beta}(-y_d)\quad &i,j<d\,\,\text{or}\,\,i=j=d,\\
			-\bar{a}_{ij}^{\alpha\beta}(-y_d)\quad &i=d,j<d\,\,\text{or}\,\,i<d, j=d.
		\end{cases}
	\end{equation}
Near the boundary of $\tilde{B}$, simply let $\tilde a_{ij}^{\alpha\beta}$ be $\delta_{ij}\delta_{\alpha\beta}$. In between, let $\tilde a_{ij}^{\alpha\beta}$ changes smoothly. Clearly, we can construct in the way that the oddness and evenness in \eqref{eqn-200702-0626} and the ellipticity condition \eqref{eqn-200701-0518} are always satisfied. To extend $\bar{f}$, we first take the even/odd extension with respective to the $y_d$-hyperplane to $B'_{cr}\times (-Mcr,0)$ for $\{\bar{f}_i\}_{i<d}$ and $\bar{f}_d$, respectively. In $\tilde B\setminus (B_{cr}'\times(-Mcr,Mcr))$, we simply take the zero extension. Now we solve
	\begin{equation}\label{eqn-200619-0700}
	\begin{cases}
	D_i(\tilde{a}_{ij}^{\alpha\beta}D_j w^{\beta}) = D_i\bar{f}_i^\alpha \quad&\text{in}\quad \tilde{B}\\
	w^\alpha = 0 \quad &\text{on}\quad \p \tilde{B}.
	\end{cases}
	\end{equation}
	According to \cite[Theorem~8.6~(iii)]{DK11b}, there is a unique solution $\vec{w}=(w^\alpha)_{\alpha=1}^m$ to \eqref{eqn-200619-0700} in $W^1_{\tilde{p}}(\tilde{B};\bR^m)$, with
	\begin{equation}\label{eqn-200619-0754}
	\norm{D\vec{w}}_{L_{\tilde{p}}(\tilde{B})} \leq N(d,m, \kappa,\tilde{p})\sum_{i,\alpha} \norm{\bar{f}_i^\alpha}_{L_{\tilde{p}(\tilde{B})}},
	\end{equation}
where the fact that the constant $N$ is independent of $r$ can be seen by using a scaling argument.
	Clearly the solution $\vec{w}$ is even in the $y^d$-variable, from which we can deduce that \eqref{eqn-200611-1218} is satisfied. Furthermore, from \eqref{eqn-200619-0754}, H\"older's inequality, Assumption \ref{ass-200604-1252} ($\theta$), and our construction,
	\begin{equation}\label{eqn-200611-0448}
	\begin{split}
	&(|D\vec{w}|^{\tilde{p}})^{1/\tilde{p}}_{\Omega_{cr}^+} \leq N (|\bar{f}_i^\alpha|^{\tilde{p}})^{1/\tilde{p}}_{\Omega_{cr}^{+}}\\
	&\leq N \left((|f_i^\alpha|^{\tilde{p}})^{1/\tilde{p}}_{\Omega_{r}} + (|1_{T(x)\in \Omega^-_{cr}} D\vu|^{\tilde{p}})^{1/\tilde{p}}_{\Omega_{cr}^{+}} + (|(\bar{a}_{ij}^{\alpha\beta} - a_{ij}^{\alpha\beta})D_ju^\beta|^{\tilde{p}})^{1/\tilde{p}}_{\Omega_{cr}^+}\right)\\
	&\leq N \left((|f_i^\alpha|^{\tilde{p}})^{1/\tilde{p}}_{\Omega_{r}} + \theta^{1/\tilde{p}-1/q} (|D\vu|^{q})^{1/q}_{\Omega_r}\right).
	\end{split}
	\end{equation}
	Now, $\vv:=\vu-\vw$ satisfies
	\begin{equation*}
	\begin{cases}
	D_i(\bar{a}_{ij}^{\alpha\beta}D_j v^{\beta}) = 0 \quad&\text{in}\quad \Omega_{cr}^{+},\\
	\bar{a}_{ij}^{\alpha\beta}D_j v^\beta n_i = 0 \quad &\text{on}\quad \Gamma_{cr}.
	\end{cases}
	\end{equation*}
	Since the problem has coefficients depending only on $y^d$, the Lipschitz estimate holds:
	\begin{equation}\label{eqn-200611-0451}
	\norm{D\vv}_{L_\infty(\Omega_{cr/2}^{+})} \leq N(d,m,\kappa,\tilde{p})
(|D\vv|^{\tilde{p}})^{1/\tilde{p}}_{\Omega_{cr}^{+}}.
	\end{equation}
This can be obtained by using \cite[Lemma~3.7]{MR3013054}, where the right-hand side is given by the $L_2$ average. To replace the $L_2$ average on the right-hand side with the $L_{\tilde{p}}$ average, we apply a standard iteration argument which involves rescaling, H\"older's inequality, and Young's inequality. See for example \cite[Lemma~8.18]{MR3099262}, or \eqref{eqn-200714-1113}-\eqref{eqn-200715-1202} below. Now, we define
	\begin{equation*}
		W:=\begin{cases}
			|D\vw|\quad &\Omega_{cr}^+,\\
			|D\vu|\quad &\Omega_{cr}^-,
		\end{cases}
		\quad \text{and}\quad
		V:=\begin{cases}
			|D\vv|\quad &\Omega_{cr}^+,\\
			0\quad &\Omega_{cr}^-.
		\end{cases}
	\end{equation*}
The estimate \eqref{eqn-200610-0250} then follows from \eqref{eqn-200611-0448} and H\"older's inequality
\begin{equation}\label{eqn-200715-1236}
	\norm{D\vec{u}}_{L_{\tilde{p}}(\Omega_{cr}^-)} \le N (\theta r^d)^{1/\tilde{p}-1/q}\norm{D\vec{u}}_{L_q(\Omega_r(x))}.
\end{equation}
The estimate \eqref{eqn-200611-0449} follows from \eqref{eqn-200611-0451}, the inequality $|D\vv|\leq |D\vu| + |D\vw|$, and \eqref{eqn-200611-0448}.

	\textbf{Case 3}: $x\notin\p\Omega$ and $B_{cr/4}(x)\cap \p\Omega \neq \emptyset$. In this case, we find a point $y\in\p\Omega$ satisfying $|x-y|< cr/4$. Let $r_1=(1-c/4)r$. By the previous case, we get the deposition of $u$ in $\Omega_{cr_1}(y)$, and $V,W$ satisfies
\begin{align*}
	(W^{\tilde{p}})^{1/\tilde{p}}_{\Omega_{cr_1}(y)} &\le N \theta^{1/\tilde{p} - 1/q}(|D\vu|^{q})^{1/q}_{\Omega_{r_1}(y)} + N(F^{\tilde{p}})^{1/\tilde{p}}_{\Omega_{r_1}(y)},\\
	\norm{V}_{L_\infty(\Omega_{cr_1/2}(y))} &\le N (|D\vu|^{q})^{1/q}_{\Omega_{r_1}(y)} + N(F^{\tilde{p}})^{1/\tilde{p}}_{\Omega_{r_1}(y)}.
\end{align*}
Finally, it remain to observe that by the triangle inequality,
$$
\Omega_{cr/8}(x)\subset \Omega_{cr_1/2}(y),\quad
\Omega_{cr/4}(x)\subset \Omega_{cr_1}(y),\quad
\Omega_{r_1}(y)\subset \Omega_r(x).
$$
The proposition is proved.
\end{proof}

\subsection{Level set argument and proof of Theorem \ref{thm-200609-0520}}\label{sec-200626-1027}
Once we have the decomposition lemma, the proof of Theorem \ref{thm-200609-0520} is more or less standard by using a level set argument. More precisely, we prove by deriving that the measure of the level set of $\mathcal{M}(|D\vu|^p)$ decays with a proper rate. Here $\mathcal{M}$ is the Hardy-Littlewood maximal function: for $f\in L_{1,loc}(\Omega)$,
\begin{equation*}
	\mM(f):=\sup_{x\in \bR^d,r>0} \fint_{B_r(x)} |f(y)|\mathbb{I}_{\Omega}\,dy.
\end{equation*}
Such idea was introduced in \cite{MR1486629} and is relied on a measure theoretical lemma in \cite{MR563790}. In this section, we give a sketch of the proof. For details, we refer the reader to \cite[Section~5]{MR3013054} and \cite[Section~5.2, 5.3]{2019arXiv190400545C}.

For two constants $(\tilde{p},q)$ with $\tilde{p}<q < p$ to be chosen later, we interpolate the $W^1_{\tilde{p}}$-estimate and the Lipschitz estimate in Proposition \ref{prop-200611-0515}, by comparing the following two level sets:
\begin{equation*}
	\mA(s):=\{x\in \Omega:\ (\mM(U^{\tilde{p}}))^{1/\tilde{p}} \geq s\}
\end{equation*}
and
\begin{equation*}
	\mB(s):= \{x\in \Omega:\ (\mM(U^{q}))^{1/{q}} + \theta^{-1/\tilde{p} + 1/{q}} (\mM(F^{\tilde{p}}))^{1/\tilde{p}} \geq s\},
\end{equation*}
where
\begin{equation*}
	U:= \sum_{i,\alpha} |D_i u^\alpha|\quad \text{and}\quad F:=\sum_{i,\alpha} |f_i^\alpha|.
\end{equation*}
The decomposition in Proposition \ref{prop-200611-0515} leads to the following stability-type result.
\begin{lemma}\label{lem-200622-1239}
	Suppose $0\in\overline{\Omega}$, under the same hypothesis of Proposition \ref{prop-200611-0515}, there exist constants $k_1$ and $N$ depending on $d$, $m$, $\kappa$, $p$, $\tilde p$, and $M$, such that for all $k>\max\{2^{d/\tilde{p}},k_1\}$ and $s>0$, the following holds: suppose for some $r<R_0$,
	\begin{equation}\label{eqn-200622-1235}
		|\Omega_{r/(32\sqrt{d}M)}\cap \mA(ks)| \geq N k^{-\tilde{p}}\theta^{1 - \tilde{p}/q}|\Omega_{r/(32\sqrt{d}M)}|,
	\end{equation}
then $\Omega_{r/(32\sqrt{d}M)} \subset \mB(s)$.
\end{lemma}
\begin{proof}
We prove the lemma by contradiction. Suppose that there exists some $y\in \Omega_{r/(32\sqrt{d}M)} \cap (\mB(s))^c$. Then by definition,
\begin{equation}\label{eqn-200624-0207}
	(\mM(U^q))^{1/q}(y) + \theta^{-1/\tilde{p} + 1/q} (\mM(F^{\tilde{p}}))^{1/\tilde{p}}(y) < s.
\end{equation}
Furthermore, we can decompose the solution according to Proposition \ref{prop-200611-0515} to obtain $V$ and $W$ on $B_{r/(4\sqrt{d}M)}$  satisfying $U\leq V+W$ on $B_{r/(4\sqrt{d}M)}$,
\begin{equation}
                                \label{eq8.32}
\norm{V}_{L_\infty(B_{r/(8\sqrt{d}M)})} \leq N_1s,\quad\text{and}\quad (W^{\tilde{p}})^{1/\tilde{p}}_{B_{r/(4\sqrt{d}M)}} < N_1
\theta^{1/\tilde{p}-1/q}s.
\end{equation}
Choose $k_1=2N_1$. Now for any point $z\in\Omega_{r/(32\sqrt{d}M)}\cap \mA(k s)$, by definition, $\mM(U^{\tilde{p}})^{1/\tilde{p}}(z)\geq k s$. Furthermore, the condition $k>2^{d/\tilde{p}}$, \eqref{eqn-200624-0207}, and H\"older's inequality guarantee that
\begin{equation}\label{eqn-200713-1237}
	\Big(\fint_{B_\tau(z)}U^{\tilde{p}}\Big)^{1/\tilde{p}}\geq k s
\end{equation}
can only occur at small scale $\tau<r/(16\sqrt{d}M)$, since otherwise because $B_\tau(z)\subset B_{2\tau}(y)$,
	\begin{equation*}
		\Big(\fint_{B_\tau(z)}U^{\tilde{p}}\Big)^{1/\tilde{p}} \leq 2^{d/\tilde{p}}\Big(\fint_{B_{2\tau}(y)}U^{\tilde{p}}\Big)^{1/\tilde{p}} \leq 2^{d/\tilde{p}} (\mM(U^q))^{1/q}(y) < ks,
	\end{equation*}
which contradicts \eqref{eqn-200713-1237}.

Now at such scale, by \eqref{eq8.32},
\begin{equation*}
	\norm{V}_{L_\infty(B_\tau(z))}\leq \norm{V}_{L_\infty(B_{r/(8\sqrt{d}M)})} \leq k_1 s/2< ks/2,
\end{equation*}
so we must have
$$(\mM(W^{\tilde{p}}\mathbb{I}_{B_{r/(4\sqrt{d}M)}}))^{1/\tilde{p}}(z) \geq \Big(\fint_{B_\tau(z)}W^{\tilde{p}}\Big)^{1/\tilde{p}} >\kappa s/2.$$
Since this is true for any $z\in \Omega_{r/(32\sqrt{d}M)}\cap\mA(ks)$, it follows from the weak-type $(1,1)$ estimate of the Hardy-Littlewood maximal function and \eqref{eq8.32} that
\begin{align*}
|\Omega_{r/(32\sqrt{d}M)}\cap \mA(ks)|&\le \Big|\big\{(\mM(W^{\tilde{p}}\mathbb{I}_{B_{r/(4\sqrt{d}M)}}))^{1/\tilde{p}}(z)>\kappa s/2\big\} \Big|\\
&\leq N\frac{\norm{W}_{L_{\tilde{p}(B_{r/(4\sqrt{d}M)})}}^{\tilde{p}}}{(ks)^{\tilde{p}}}\\
& < N k^{-\tilde{p}}\theta^{1 - \tilde{p}/q}|\Omega_{r/(32\sqrt{d}M)}|,
\end{align*}
which contracts \eqref{eqn-200622-1235} and finishes the proof of the lemma.
\end{proof}

From Lemma \ref{lem-200622-1239}, we can prove the following decay estimate by the ``crawling of ink spots'' lemma.
\begin{corollary}\label{cor-200624-0227}
 	Under the same hypothesis of Proposition \ref{prop-200611-0515}, there exists a constant $N=N(d,m,\kappa,q,\tilde p,M)>0$ such that for $k>\{2^{d/p},k_1,k_2\}$ and
	\begin{equation}\label{eqn-200624-0108}
		s> s_0:= N\theta^{1/q-1/\tilde{p}}\norm{U}_{L_{\tilde{p}}}/|B_{R_0/(32\sqrt{d}M)}|^{1/\tilde{p}},
	\end{equation}
	we have
	\begin{equation*}
	|\mA(ks)| \leq N k^{-\tilde{p}}\theta^{1 - \tilde{p}/q}|\mB(s)|,
	\end{equation*}
were $k_1$ is the constant in Lemma \ref{lem-200622-1239} and $k_2$ is a constant such that
	\begin{equation}\label{eqn-200624-0159}
		Nk_2^{-\tilde{p}}\theta^{1/\tilde{p}-1/q}<1/3.
	\end{equation}
\end{corollary}
Such measure theoretical lemma in \cite[Section~2]{MR563790}. Or, one can refer to the proof of \cite[Corollary~5.7]{2019arXiv190400545C} via a stopping time argument. Note that the condition \eqref{eqn-200624-0108} allows us to start the stopping time argument at the radius $r=R_0$, while combined with the Lebesgue differentiation theorem, \eqref{eqn-200624-0159} means that such argument would stop at finite time (radius).

As the last step, the $L_p$ estimate follows by estimating
\begin{equation}\label{eqn-200703-1217}
	\int_0^\infty |\mA(s)|s^{p-1}\,ds.
\end{equation}
Let us now choose some parameters: fix $\tilde{p}=(p+2)/3$ and $q=(p+1)/2$ both lying between $1$ and $p$, and then we fix $k$ to be some constant satisfying the condition in Corollary \ref{cor-200624-0227} with $\theta=1$, so that it is satisfied for any $\theta\le 1$. For \eqref{eqn-200703-1217}, we take the change of variable $s\mapsto ks$ and then split the integral into two parts. For $s\leq s_0$, we apply the weak-$(1,1)$ estimate. For $s>s_0$, we apply Corollary \ref{cor-200624-0227}. From these, we obtain
\begin{align*}
 \int_0^\infty |\mA(s)|s^{p-1}\,ds \leq N_1 R_0^{d(\tilde{p}-p)/\tilde{p}}\norm{U}_{L_{\tilde{p}}}^p + N_2\int_0^\infty k^{p-\tilde{p}}\theta^{1-\tilde{p}/q}|\mB(s)|s^{p-1}\,ds,
\end{align*}
where $N_1=N_1(d,m,\kappa,p,M,k,\theta)$ and $N_2=N_2(d,m,\kappa,p,M)$. Splitting $\mB(s)$ as
\begin{equation*}
	\mB(s)\subset \{(\mM(U^q))^{1/q}>s/2\}\cup \{\theta^{-1/\tilde{p}+1/q} (\mM(F^{\tilde{p}}))^{1/\tilde{p}} \geq s/2\}
\end{equation*}
and using the representation formula for the $L_p$ norm, we obtain
\begin{multline*}
\norm{\mM(U^{\tilde{p}})}_{L_{p/\tilde{p}}}^{p/\tilde{p}} \leq N_3\left(R_0^{d(\tilde{p}-p)/\tilde{p}}\norm{U}_{L_{\tilde{p}}}^p + \norm{\mM(F^{\tilde{p}})}_{L_{p/\tilde{p}}}^{p/\tilde{p}}\right)\\
+ N_4k^{p-\tilde{p}}\theta^{1-\tilde{p}/q}\norm{\mM(U^q)}_{L_{p/q}}^{p/q},
\end{multline*}
where $N_3=N_3(d,m,\kappa,p,M,k,\theta)$ and $N_4=N_4(d,m,\kappa,p,M)$.
Now, noting that $U\in L_p(\Omega)$ and $\tilde{p}<q<p$, we can use the Hardy-Littlewood maximal function estimate to obtain
\begin{equation*}
		\norm{U}_{L_p}^p \leq \tilde{N}_3\left(R_0^{d(\tilde{p}-p)/\tilde{p}}\norm{U}_{L_{\tilde{p}}(\Omega)}^p + \norm{F}_{L_p}^p\right) + \tilde{N}_4k^{p-\tilde{p}}\theta^{1-\tilde{p}/q}\norm{U}_{L_p}^p.
\end{equation*}
By H\"older's inequality and Young's inequality, we can replace the $L_{\tilde{p}}$ norm on the right-hand side by
\begin{equation*}
		\tilde{N}_3R_0^{d(1-p)}\norm{U}_{L_1(\Omega)}^p + \frac{1}{3}\norm{U}_{L_p}^p,
\end{equation*}
where $\tilde{N}_3$ is another constant still depending on $(d,m,\kappa,p,M,k,\theta)$. Finally, we absorb the $L_p$ norm  on the right-hand side by choosing $\theta$ small enough, noting that $\tilde{N}_4$ is independent of $\theta$.

\subsection{Solvability and proof of Theorem \ref{thm-200609-0525}}\label{sec-200705-1036}
Before we prove Theorem \ref{thm-200609-0525}, let us first show how to obtain Remark \ref{rmk-200619-1150} (b). For $p>2$, actually having the Sobolev-Poincar\'e inequality \eqref{eqn-200619-0355} in hand, under the assumptions of Theorem \ref{thm-200609-0520}, instead of the a priori estimate, we can prove a regularity result: any solution $\vu\in\dot{L}^1_2$ with $\vf_i\in L_p(\Omega)$ is also in $\dot{L}^1_p$. The proof is almost the same as that of Theorem \ref{thm-200609-0520}. First, by \eqref{eqn-200619-0355}, Gehring's lemma, and the standard $L_2$ estimate, we can derive the following reverse H\"older inequality: for any $x\in\overline{\Omega}$ and $r<R_0$,
\begin{equation}\label{eqn-200620-0607}
	(|D\vu|^{2\mu})^{1/(2\mu)}_{\Omega_{r/2}(x)} \leq N \left((|D\vu|^{2})^{1/2}_{\Omega_{r}(x)} + |F|^{2\mu})^{1/(2\mu)}_{\Omega_{r}(x)}\right),
\end{equation}
where $\mu>1$ is a constant depending on $d$, $m$, $\kappa$, $p$, $\epsi$, and $N_{\Omega,\epsi}$. The proof of such reverse H\"older inequality is classical and can be found in \cite[Section~6.5]{MR3099262} and \cite[Lemma~3.3]{2019arXiv190400545C}. This improves the regularity by a little bit, from $\dot{L}^1_2$ to $\dot{L}^1_{2\mu}$. To further improve the regularity all the way up to $\dot{L}^1_p$, we apply the level set argument. We first apply Proposition \ref{prop-200611-0515} with $\tilde{p}=2$ and $q=2\mu$ to find $V$ and $W$. Then, by \eqref{eqn-200620-0607}, the estimates \eqref{eqn-200610-0250} and \eqref{eqn-200611-0449} can be improved:
\begin{equation*}
(W^2)^{1/2}_{\Omega_{r/(4\sqrt{d}M)}(x)} \le N \theta^{1/2 - 1/(2\mu)}(|D\vu|^{2})^{1/2}_{\Omega_r(x)} + N(F^{2\mu})^{1/(2\mu)}_{\Omega_{r}(x)},
\end{equation*}
\begin{equation*}
\norm{V}_{L_\infty(\Omega_{r/(8\sqrt{d}M)}(x))} \le N(|D\vu|^2)^{1/2}_{\Omega_r(x)} + N(F^{2\mu})^{1/(2\mu)}_{\Omega_{r}(x)}.
\end{equation*}
From this, the aforementioned level set argument shows that $\vu\in \dot{L}^1_p(\Omega)$ and
\begin{equation*}
	\norm{D\vu}_{L_p(\Omega)} \leq N(d,m,\kappa,p,\epsi,N_{\Omega,\epsi})\left( R_0^{d(1/p-1/2)}\norm{D\vu}_{L_2(\Omega)} + \norm{F}_{L_p(\Omega)}\right).
\end{equation*}
Furthermore, the $\norm{D\vu}_{L_2(\Omega)}$ term on the right-hand side can be dropped by the $L_2$ estimate, H\"older's inequality, and the condition $|\Omega|<\infty$.

Once we have the unique solvability and estimate for $p>2$, the result for $p\in(1,2)$ can be obtained by duality.

Now we are in the position of proving Theorem \ref{thm-200609-0525}.

\begin{proof}[Proof of Theorem \ref{thm-200609-0525}]
	We first deal with the nonzero $\vvarphi$ and $\vg$. Let $p'$ be the conjugate exponent of $p$: $1/p'+1/p=1$ and $(p^*)'$ be the conjugate exponent of $p^*$. When $p^*=1$, $(p^*)'=\infty$.
For each $\alpha$, consider the linear functional $\ell$, defined by
	\begin{equation*}
		\ell(\phi) := \int_{\p\Omega} \phi\varphi^\alpha - \int_{\Omega}\phi g^\alpha.
	\end{equation*}
Due to the embeddings
\begin{equation*}
	W^1_{p'}(\Omega) \hookrightarrow W^{1-1/p'}_{p'}(\p\Omega)\quad \text{and}\quad W^1_{p'}(\Omega)\hookrightarrow L_{(p^*)'}(\Omega)
\end{equation*}
and the compatibility condition
\begin{equation*}
	\int_{\p\Omega} \varphi^\alpha = \int_{\Omega} g^\alpha,
\end{equation*}
we know that $\ell$ is well-defined and bounded on the homogeneous space $\dot{L}^1_{p'}(\Omega)$ equipped with the norm $\sum_i\norm{D_i \phi}_{L_{p'}(\Omega)}$. According to \cite[Theorem~II.8.2]{MR2808162}, which is obtained by the Hahn-Banach theorem and the Riesz representation theorem, we can find functions $h_i^\alpha \in L_p(\Omega)$, such that
\begin{equation*}
	\ell(\phi) = \int_\Omega h_i^\alpha D_i \phi.
\end{equation*}
Furthermore, we have the bound
\begin{equation*}
	\sum_i\norm{h_i^\alpha}_{L_p(\Omega)} = \norm{\ell}_{(\dot{L}^1_{p'}(\Omega))'} \le N \norm{\varphi^\alpha}_{W^{-1/p}_p(\p\Omega)} + N\norm{g^\alpha}_{L_{p^*}(\Omega)}.
\end{equation*}
Noting the weak formulation \eqref{eqn-200703-0330} and the fact that on a Lipschitz domain the space $\dot{L}^1_{p'}$ can be identified with the space $W^1_{p'}(\Omega)\cap\{\int_\Omega \phi =0\}$, now we only need to deal with \eqref{eqn-200622-1036}-\eqref{eqn-200609-0514} with $\vg=\vvarphi=0$, and $\vf_i$ replaced with $\vf_i + \vec{h}_i$.

After this reduction, the situation is the same with Theorem \ref{thm-200609-0520}. The estimate and the solvability follows from Remark \ref{rmk-200619-1150}.
\end{proof}

\section{Weighted estimate for Robin boundary value problems}
                            \label{sec4}
\subsection{Properties of \texorpdfstring{$A_p$}{Ap} weights and an interpolation lemma}
In this section, we summarize some properties of $A_p$ weights which will be used in this paper. The properties in Lemma \ref{lem-200615-1006} are standard. See for example, \cite{MR2445437}.
\begin{lemma}\label{lem-200615-1006}
	Suppose that $\omega\in A_p$ for some $p\in(1,\infty)$. There exist constants $s$ and $\hat{p}$ depending on $d$, $p$, and $[\omega]_{A_p}$, such that the following hold.
	\begin{enumerate}[wide, labelwidth=0pt, labelindent=0pt]
		\item There exists a constant $s>1$ depending on $d$, $p$, and $[\omega]_{A_p}$, such that for any ball $B\subset \bR_d$,
		\begin{equation}\label{eqn-200616-0917}
			\left(\fint_B \omega(x)^{s}\,dx\right)^{1/s} \leq N \fint_B \omega(x)\,dx,
		\end{equation}
		where $N=N(d,p,[\omega]_{A_p})$. A weight $\omega$ satisfying  \eqref{eqn-200616-0917} is said to belong to the reverse H\"older class $RH_s$, where the optimal constant $N$ is usually denoted to be $[\omega]_{RH_s}$.
		\item There exists a constant $\hat{p}\in(1,p)$ depending on $d$, $p$, and $[\omega]_{A_p}$, such that $\omega\in A_{p/\hat{p}}$.
		\item Let $p'$ be the conjugate exponent of $p$, namely $1/p+1/p'=1$, then $\omega^{-p'/p} \in A_{p'}$, where $[\omega^{-p'/p}]_{A_{p'}} = [\omega]^{p'/p}_{A_p}$.
		\item For any constant $\delta\in(0,1)$, $\omega^\delta \in A_{\delta p + 1 - \delta}$ with
		\begin{equation*}
			[\omega^\delta]_{A_{\delta p + 1 - \delta}} \leq [\omega]_{A_p}^\delta.
		\end{equation*}
	\end{enumerate}
\end{lemma}

The second lemma includes two embedding results for weighted spaces.
\begin{lemma}\label{lem-200624-0616}
	Suppose that $\omega\in A_p$ for some $p\in(1,\infty)$.
	\begin{enumerate}[wide, labelwidth=0pt, labelindent=0pt]
	\item For any $q\in [p,\infty)$, we have the embedding
	$$
L_{qs/(s-1)}\hookrightarrow L_{q,\omega}\hookrightarrow L_{q/p},
$$
where $s>0$ can be any constant satisfying \eqref{eqn-200616-0917}.
	\item For any $u\in W^1_{p,\omega}$, we have the following Poincar\'e type inequality
	\begin{equation*}
		\norm{u-c}_{L_{np/(n-1),\omega}(\Omega)} \leq N(d,p,M ,[\omega]_{A_p})|\Omega|^{1/n}|\omega(\Omega)|^{-1/(np)}
\norm{Du}_{L_{p,\omega}(\Omega)},
	\end{equation*}
	where
$$
c=\fint_\Omega u\quad \text{or}\quad
\frac 1{ \omega(\Omega)}\int_\Omega u\omega.
$$
	\end{enumerate}
\end{lemma}
The first embedding can be proved simply by H\"older's inequality and the definition of $A_p$ weights. The second embedding was established in \cite[Theorem~1.5]{MR643158}. See also \cite[Lemma~5.9]{YYY20}.

The idea of proving the $W^1_{p,\omega}$ well-posedness is again by interpolation. Compared to Section \ref{sec-200626-1027}, here we need a modified level set argument which allows us to include the weight $\omega$. For brevity, we state a ready-to-use interpolation theorem, which can be simply derived from \cite[Theorem~3.1]{YYY20}. Recall from Lemma \ref{lem-200615-1006} that for any $A_p$-weight $\omega$, we can find $\hat{p}\in(1,p)$ and $s>1$ depending on $d,p,[\omega]_{A_p}$ such that $\omega\in A_{p/\hat{p}}$ and $\omega\in RH_s$.

\begin{theorem}\label{thm-200604-0840}
	Let $\Omega\subset \bR^d$ be a bounded Lipschitz domain with parameters $R_0$ and $M$, $p\in(1,\infty)$, $\omega\in A_p$, and $F$ be a non-negative function defined on $\Omega$. Let $p_1$, $p_2$, $p_3$, and $\nu$ be parameters satisfying
	\begin{equation*}
		p_1<p, \quad \nu\in[0,1), \quad p_3>\frac{ps}{s-1},\quad\text{and}\quad \frac{p_2}{1-\nu}\leq \hat{p}.
	\end{equation*}
Suppose that $U\in L_{p_1}(\Omega)$ is a non-negative function such that for any $x\in\overline{\Omega}$ and any $r<R_0$, we have a non-negative decomposition $U \leq V + W$ in $\Omega_r(x)$ with estimates
	\begin{equation*}
	\Big(\fint_{\Omega_r(x)} V^{p_3}\Big)^{1/p_3} \leq C_1 \left(\fint_{\Omega_{2r}(x)} U^{p_1}\right)^{1/p_1} + C_1 \left(|\Omega_{2r}(x)|^\nu\fint_{\Omega_{2r}(x)} F^{p_2}\right)^{1/p_2}
	\end{equation*}
	and
	\begin{equation*}
	\Big(\fint_{\Omega_r(x)} W^{p_1}\Big)^{1/p_1} \leq C_2 \left(|\Omega_{2r}(x)|^\nu\fint_{\Omega_{2r}(x)} F^{p_2}\right)^{1/p_2},
	\end{equation*}
	where $C_1$ and $C_2$ are constants independent of $U$, $F$, $x$, and $r$. Then if $F\in L_{p_0,\omega^{p_0/p}}$, we have $U\in L_{p,\omega}$ with
	\begin{equation*}
	\left[\frac{1}{\omega(\Omega)}\int_\Omega U^p\omega\right]^{\frac 1 p} \leq N \left(\frac{1}{|\Omega|}\int_\Omega U^{p_1}\right)^{\frac 1 {p_1}} + N \left(\frac{1}{\omega(\Omega)}\int_\Omega F^{p_0}\omega^{\frac {p_0} p}\omega\right)^{\frac 1 {p_0}}\left[\omega(\Omega)\right]^{\frac \nu {p_2}}.
	\end{equation*}
Here $p_0$ is a constant given by $1/p_0 = 1/p + \nu/p_2$ and $N$ is a constant depending only on $(R_0,M, C_1,C_2,d,p_1,p_2,p,[\omega]_{A_p})$.
\end{theorem}

\subsection{Unweighted \texorpdfstring{$L_p$}{Lp} estimate and a reverse H\"older inequality}
As a preparation, we first prove the unweighted estimate. In this case, we consider a more general Robin boundary condition
\begin{equation}
                                        \label{eq2.05}
    a_{ij}^{\alpha\beta}D_j u^\beta n_i+ \gamma^{\alpha\beta} u^\beta=f_i^\alpha n_i+\varphi^\alpha\quad \text{on}\quad \partial \Omega,
\end{equation}
where $\varphi^\alpha\in W^{-1/p}_p(\p\Omega)$.
\begin{theorem}\label{thm-200614-1046}
	Under the assumptions of Theorem \ref{thm-200527-1035}, we have the $W^1_p$ well-posedness of \eqref{eqn-200622-1036}-\eqref{eq2.05} for any $p\in(1,\infty)$. In other words, for any $\vf_i\in L_p$, $\vg\in L_{p^*}$, and $\vvarphi\in W^{-1/p}_p(\p\Omega)$. there exists a unique solution $\vu\in W^1_p(\Omega;\bR^m)$, satisfying
	\begin{equation*}
	\|\vu\|_{W^1_p(\Omega)}\le N
\Big(\sum_i\|\vf_i\|_{L_p(\Omega)}+\|\vg\|_{L_{p^*}(\Omega)}+\|\vvarphi\|_{W^{-1/p}_p(\p\Omega)}\Big),
	\end{equation*}
where $N=N(d,m,\kappa,M,p,R_0,\gamma_0,\Omega,E)>0$.
\end{theorem}
We prove this by using the estimate for the conormal problem.
\begin{proof}
	When $p=2$, the result is standard and no smallness condition on $\theta$ is needed. It follows from the Lax-Milgram lemma, the Sobolev embedding $W^1_2(\Omega)\hookrightarrow L_{2d/(d-2)}(\Omega)$, the trace lemma, and the Friedrichs inequality: for $v\in W^1_2$ and $E\subset \p\Omega$ with $|E|>0$, there exists some constant $N=N(d,E,\Omega)$, such
	\begin{equation*}
		\int_\Omega |v|^2 \leq N\left(\int_\Omega |Dv|^2 + \int_E|v|^2\right).
	\end{equation*}
The above Friedrichs inequality is classical. See, for example, \cite[Appendix]{MR2038145} and \cite{MR3429656}.
	
	For the case when $p>2$, we apply a bootstrap argument. Let $\theta$ be small enough such that Theorem \ref{thm-200609-0525} with the integrability exponent between $2$ and $p$ holds. First, since $\vf_i\in L_p(\p\Omega)\subset L_2(\p\Omega)$, $\vg\in L_{p^*}(\Omega)\subset L_{2^*}(\Omega)$, and $\vvarphi\in W^{-1/p}_p(\p\Omega)\subset W^{-1/2}_2(\p\Omega)$,
we can solve for a weak solution $\vu\in W^1_2(\Omega)$. By the trace theorem and embedding,
	$$
	\vu\in W^1_2(\Omega) \hookrightarrow W^{1/2}_2(\p\Omega) \hookrightarrow W^{-1/q}_q(\partial\Omega),
	$$
	where $q>2$ satisfying $-1/2 + (d-1)/2 = 1/q + (d-1)/q$. Using Theorem \ref{thm-200609-0525}, we can improve the regularity of $\vu$ to $W^1_{\min\{p,q\}}(\Omega)$ with the corresponding estimate. Repeating this if needed, in finite steps we reach the required $W^1_p$ regularity.
	
	As usual, the case when $p\in (1,2)$ can be obtained by duality. The theorem is proved.  
\end{proof}
\begin{remark}\label{rmk-200709-0752}
		For the unweighted $W^1_p$ solvability and estimate, the boundedness on $\gamma$ can be generalized to
		\begin{equation*}
			\gamma\in \begin{cases}
				L_{p(d-1)/d}\quad &\text{when}\,\,p>d\\
				L_{d-1 +\epsi} &\text{when}\,\,p\leq d,
			\end{cases}
		\end{equation*}
where $\epsi$ can be any positive number.
\end{remark}
Next, we prove a local result: a reverse H\"older inequality for the Robin problem with homogeneous right-hand side. This plays the key role in the proof of Theorem \ref{thm-200527-1035}. Compared to the localization argument for conormal problems, here we do not have the freedom of subtracting any constant $c$ from a solution. In the proof, we instead subtract a function which is obtained by solving a conormal problem with a inhomogeneous boundary condition.
\begin{lemma}
	\label{lemma1}
	Let $p\in (1,\infty)$, $x_0\in \partial\Omega$, and $r\in (0,R_0/2)$. Suppose that $\vu\in W^1_{p_0}(\Omega)$ is a weak solution to \eqref{eqn-200622-1036}-\eqref{eqn-200527-1002} with $p_0>1$ and $\vf_i\equiv \vg\equiv 0$ in $\Omega_{2r}(x_0)$. Then we have
	$$
	(|\vu|^p)^{1/p}_{\Omega_r(x_0)}+(|D\vu|^p)^{1/p}_{\Omega_r(x_0)}
	\le N(|\vu|)_{\Omega_{2r}(x_0)}+N(|D\vu|)_{\Omega_{2r}(x_0)},
	$$
	where $N>0$ depends on $(d,m,\kappa,M,p,R_0,|\Omega|,\gamma_0)$.
\end{lemma}
\begin{proof}
Without loss of generality, we assume that $x_0=0$ and $p>2$. By using a bootstrap argument similar to the proof of Theorem \ref{thm-200614-1046}, we know that $u\in W^1_p(\Omega_{r_1})$ for any $r_1\in (0,2r)$. In the following, for $r>0$, denote $$\Lambda_r:=\p\Omega\cap B_r.$$
	We take a smooth cutoff function $\zeta\in C_0(B_{2r})$ such that $\zeta=1$ on $B_{3r/2}$. 
	Take another function $\varphi^\alpha\in L_p(\partial\Omega)$ such that $\supp(\varphi^\alpha)\subset\Lambda_{3r}$, $\varphi^\alpha=\gamma^{\alpha\beta} u^{\beta}\zeta$ on $\Lambda_{2r}$,
	$$
	\int_{\partial \Omega} \varphi^\alpha=0, \quad\text{and}\quad \|\varphi^\alpha\|_{L_p(\Lambda_{3r})}\le N\|\gamma^{\alpha\beta} u^\beta\zeta\|_{L_p(\Lambda_{2r})}.
	$$
	Denote $\vec{\varphi}:=(\varphi^\alpha)_\alpha$. Let $\vec{w}=(w^\alpha)_\alpha\in W^1_p(\Omega;\bR^m)$ be a solution to
	$$
	D_i(a_{ij}^{\alpha\beta}D_j w^\beta)=0\quad\text{in}\quad \Omega
	$$
	with the inhomogeneous conormal boundary condition
	$$
	a_{ij}^{\alpha\beta}D_j w^\beta n_i=-\varphi^\alpha\quad\text{on}\quad \partial\Omega
	$$
	and satisfies $(\vec{w})_{\Omega_{3r/2}}=(\vu)_{\Omega_{3r/2}}$. According to Theorem \ref{thm-200609-0525}, such solution exists and satisfies
	$$
	\|D\vec{w}\|_{L_p(\Omega)}\le N\|\vec{\varphi}\|_{W^{-1/p}_p(\p\Omega)}.
	$$
	Noting that $\vvarphi$ has mean zero and has support on $\Lambda_{3r}$,	we have
	\begin{align}
			\|\vec{\varphi}\|_{W^{-1/p}_p(\p\Omega)} &= \sup_{\norm{\vec{\phi}}_{\dot{W}^{1/p}_{p'}(\Lambda_{3r})}=1}\left|\int_{\Lambda_{3r}} \vvarphi\cdot\vec{\phi}\right| = \sup_{\vec{\phi}}\left|\int_{\Lambda_{3r}} \vvarphi\cdot(\vec{\phi} - (\vec{\phi})_{\Lambda_{3r}})\right|\nonumber\\
			&\leq N\sup_{\vec{\phi}} \norm{\vvarphi}_{L_p(\Lambda_{3r})} \norm{\vec{\phi}- (\vec{\phi})_{\Lambda_{3r}}}_{L_{p'}(\Lambda_{3r})}\nonumber\\
			&= Nr^{1/p}\norm{\vvarphi}_{L_p(\Lambda_{3r})},\label{eqn-200706-0358}
	\end{align}
where $\norm{\cdot}_{\dot{W}^{1/p}_{p'}}$ is the usual homogeneous semi-norm for the fractional Sobolev space. In order to obtain \eqref{eqn-200706-0358}, we used the following fractional Sobolev inequality
	\begin{equation}\label{eqn-200627-0438} \|f-(f)_{\Lambda_{3r}}\|_{L_{p'}
(\Lambda_{3r})}\leq N(d,M,p',s)r^{s}\|f\|_{\dot{W}^s_{p'}(\Lambda_{3r})}
	\end{equation}
	and H\"older's inequality. The inequality \eqref{eqn-200627-0438} can be found, for example, in \cite{MR1945278}. Hence,
	$$\norm{D\vec{w}}_{L_p(\Omega)}	\le Nr^{1/p}\|\vec{\varphi}\|_{L_p(\Lambda_{3r})}
	\le N\sum_\alpha r^{1/p}\|\gamma^{\alpha\beta} u^\beta\zeta\|_{L_p(\Lambda_{2r})}.
	$$
Since $(\vec{w})_{\Omega_{3r/2}}=(\vu)_{\Omega_{3r/2}}$, by the Poincar\'e inequality and the trace lemma with a scaling, for any $\varepsilon\in (0,1)$,
	we have
	\begin{align}
		&\|\vec{w}\|_{W^{1}_p(\Omega_{3r/2})}\nonumber\\
&\le \|D\vec{w}\|_{L_p(\Omega_{3r/2})}
+\|\vec{w}-(\vec{w})_{\Omega_{3r/2}}\|_{L_p(\Omega_{3r/2})}
+\|(\vec{w})_{\Omega_{3r/2}}\|_{L_p(\Omega_{3r/2})}\nonumber\\
&\le Nr^{1/p}\|\gamma^{\alpha\beta} u^\beta\zeta\|_{L_p(\partial\Omega)} + Nr^{d/p-d}\|\vu\|_{L_1(\Omega_{3r/2})}\nonumber\\
&\leq Nr^{1/p}(\norm{D\vu}_{L_q(\Omega_{2r})} + r^{-1}\norm{\vu}_{L_q(\Omega_{2r})}) + Nr^{d/p-d}\|\vu\|_{L_1(\Omega_{3r/2})}\nonumber\\
		&\le \varepsilon(r\|D\vu\|_{L_p(\Omega_{2r})}
		+\|\vu\|_{L_p(\Omega_{2r})})
		+N(\varepsilon)r^{d/p-d}\Big(r\|D\vu\|_{L_1(\Omega_{2r})}+
\|\vu\|_{L_1(\Omega_{2r})}\label{eqn-200703-0513}\Big)\\
		&\le \varepsilon r\|D\vu\|_{L_p(\Omega_{2r})}
		+N(\varepsilon)r^{d/p-d}\big(r\|D\vu\|_{L_1(\Omega_{2r})}
+\|\vu\|_{L_1(\Omega_{2r})}\big),\label{eq4.08}
	\end{align}
where $q>1$ satisfies $$
\frac 1q =\frac 1 d\Big(\frac{d-1} p+1\Big).
$$
Here in \eqref{eqn-200703-0513} we also used H\"older's inequality and Young's inequality, and in \eqref{eq4.08} we used the interpolation inequality.	

Let $\vec{v}:=\vu-\vec{w}$, which satisfies the homogeneous equation
	$$
	D_i(a_{ij}^{\alpha\beta}D_j v^\beta)=0\quad\text{in}\quad \Omega_{3r/2}
	$$
	with the homogeneous conormal boundary condition
	$$
	a_{ij}^{\alpha\beta}D_j v^\beta n_i=0\quad\text{on}\quad \Lambda_{3r/2}.
	$$
We can localize the estimate in Theorem \ref{thm-200609-0525} to obtain the estimate for $\vv$. For $B_r\subset B_s\subset B_t\subset B_{3r/2}$, let $\eta\in C^\infty_c(B_s)$ be a usual cut-off function with $\eta=1$ on $B_s$ and $|D\eta|\le  N/(t-s)$. Then $\vec{v}\eta$ solves
	\begin{equation}\label{eqn-200710-0150}
		D_i(a_{ij}^{\alpha\beta}D_j(v^\beta \eta)) = D_i(a_{ij}^{\alpha\beta}v^\beta D_j\eta) + a_{ij}^{\alpha\beta}D_j v^\beta D_i\eta\quad\text{in}\quad \Omega
	\end{equation}
	with the homogeneous conormal boundary condition on $\partial\Omega$. By Theorem \ref{thm-200609-0525} applied to \eqref{eqn-200710-0150} and our construction of $\eta$,
	\begin{equation*}
		\norm{D\vec{v}}_{L_p(\Omega_s)} \leq N\left(\frac{1}{t-s}\norm{\vec{v}}_{L_p(\Omega_t\setminus \Omega_s)} + \frac{1}{t-s}\norm{D\vec{v}}_{L_{p^*}(\Omega_t\setminus \Omega_s)}\right).
	\end{equation*}
Applying H\"older's inequality and Young's inequality to the second term on the right-hand side, we obtain that for some small $\epsi$ to be determined later,
\begin{align}\label{eqn-200714-1113}
	\norm{D\vec{v}}_{L_p(\Omega_s)} &\leq \epsi\norm{D\vec{v}}_{L_p(\Omega_t\setminus\Omega_s)} + N\frac{1}{t-s}\norm{\vec{v}}_{L_p(\Omega_t\setminus \Omega_s)}\nonumber\\
&\qquad + N(\epsi)(t-s)^{-d+d/p}\norm{D\vec{v}}_{L_1(\Omega_t\setminus\Omega_s)}.
\end{align}
We absorb the first term on the right-hand side by an iteration argument. Take an increasing sequence $r_k:=3r/2-2^{-k}r$. Note that $r_1=r$ and $r_\infty=3r/2$. Now taking $t=r_{k+1}$ and $s=r_k$ in \eqref{eqn-200714-1113}, we obtain
\begin{equation*}
	\begin{split}
	\norm{D\vec{v}}_{L_p(\Omega_{r_k})} &\leq \epsi\norm{D\vec{v}}_{L_p(\Omega_{r_{k+1}})}+ N2^{k+1}r^{-1}\norm{\vec{v}}_{L_p(\Omega_{r_{k+1}}\setminus \Omega_{r_k})}
\\
	&\quad + N(\epsi)(2^{k+1}r^{-1})^{d-d/p}\norm{D\vec{v}}_{L_1(\Omega_{r_{k+1}}\setminus\Omega_{r_k})}.
	\end{split}
\end{equation*}
Choosing $\epsi$ small enough such that $2^{d-d/p}\epsi<1$, multiplying both sides by $\varepsilon^k$, then adding up in $k$, we have
\begin{equation}\label{eqn-200715-1202}
	\norm{D\vec{v}}_{L_p(\Omega_r)} \leq Nr^{-1}\norm{\vec{v}}_{L_p(\Omega_{3r/2})} + Nr^{-d+d/p}\norm{D\vec{v}}_{L_1(\Omega_{3r/2})}.
\end{equation}
To estimate the first term on the right-hand side of \eqref{eqn-200715-1202}, by using $(\vec{v})_{\Omega_{3r/2}}=0$, the Sobolev-Poincar\'e inequality, H\"older's inequality and Young's inequality, for any small constant $\epsi>0$,
\begin{align}\label{eqn-200706-0732}
r^{-1}(|\vec{v}|^p)^{1/p}_{\Omega_{3r/2}} \leq N(|D\vec{v}|^{p^*})^{1/p^*}_{\Omega_{3r/2}} \leq \epsi(|D\vec{v}|^{p})^{1/p}_{\Omega_{3r/2}} + N(\epsi)(|D\vec{v}|)_{\Omega_{3r/2}}.
\end{align}
From \eqref{eqn-200706-0732} and \eqref{eqn-200715-1202}, we obtain
	\begin{equation}		\label{eq4.09}
	\begin{split}
	&r^{-1}(|\vec{v}|^p)^{1/p}_{\Omega_r}+(|D\vec{v}|^p)^{1/p}_{\Omega_r}
	\le \epsi(|D\vec{v}|^{p})^{1/p}_{\Omega_{3r/2}} + N(\epsi)(|D\vec{v}|)_{\Omega_{3r/2}}\\
	&\le \epsi(|D\vec{w}|^{p})^{1/p}_{\Omega_{3r/2}} + \epsi(|D\vec{u}|^{p})^{1/p}_{\Omega_{3r/2}} + N(\epsi)(|D\vec{v}|)_{\Omega_{3r/2}}.
	\end{split}
	\end{equation}
	Combining \eqref{eq4.08} and \eqref{eq4.09}, by the triangle inequality we obtain
\begin{align*}
&	(|\vu|^p)^{1/p}_{\Omega_r}+(|D\vu|^p)^{1/p}_{\Omega_r}\\
&	\le N(\varepsilon)(|\vu|)_{\Omega_{2r}}+N(\varepsilon)(|D\vu|)_{\Omega_{2r}}
	+N\varepsilon (r+1)(|D\vu|^p)^{1/p}_{\Omega_{2r}}.
\end{align*}
Finally, an iteration argument gives the desired estimate.
\end{proof}

\subsection{Proof of Theorem \ref{thm-200527-1035}}

Now we turn to the proof of Theorem \ref{thm-200527-1035}.

\textbf{Proof of Case (a)}. Since the estimates involving $\vf$ and $\vg$ require versions of Theorem \ref{thm-200604-0840} with different parameters, we treat them separately. For some $p_1>1$ to be chosen later, by Theorem \ref{thm-200614-1046}, we solve for $u_1^\alpha$ and $u_2^\alpha$, which are $W^1_{p_1}$ weak solutions to
\begin{equation*}
\begin{cases}
D_i(a_{ij}^{\alpha\beta}D_j u_1^{\beta})=D_i f_i^\alpha\quad&\text{in}\quad \Omega,\\
a_{ij}^{\alpha\beta}D_j u_1^\beta n_i+ \gamma^{\alpha\beta} u_1^\beta=f_i^\alpha n_i\quad &\text{on}\quad \partial \Omega
\end{cases}
\end{equation*}
and
\begin{equation}\label{eqn-200604-0843-2}
\begin{cases}
D_i(a_{ij}^{\alpha\beta}D_j u_2^{\beta})=g^\alpha\quad&\text{in}\quad \Omega,\\
a_{ij}^{\alpha\beta}D_j u_2^\beta n_i+ \gamma^{\alpha\beta} u_2^\beta=0 \quad &\text{on}\quad \partial \Omega.
\end{cases}
\end{equation}

We will show that $u^\alpha_{1}, u^\alpha_{2}\in W^1_{p,\omega}$ and satisfy the desired estimates. We first estimate $u_1^\alpha$. The interpolation in Theorem \ref{thm-200604-0840} requires the further decomposition of $u^\alpha_{1}$: the ``$V$'' part which satisfies the self-improving property -- the reverse H\"older inequality, and the ``$W$'' part which is small. More precisely, for any $x_0\in \partial\Omega$, let $\zeta$ be a cut-off function supported on $B_{2r}(x_0)$ with $\zeta=1$ on $B_r(x_0)$. Now solve for $w_1^\alpha\in W^{1}_{p_1}(\Omega)$, which is the unique solution to
$$
D_i(a_{ij}^{\alpha\beta}D_j w_1^\beta)=D_i (f_i^\alpha\zeta)\quad \text{in}\quad \Omega
$$
with the boundary condition
$$
a_{ij}^{\alpha\beta}D_j w_1^\beta n_i + \gamma^{\alpha\beta} w_1^\beta=f_i^\alpha\zeta n_i\quad \text{on}\quad \partial \Omega,
$$
satisfying
$$
(|\vw_1|^{p_1})^{1/p_1}_{\Omega_{2r}(x_0)}
+(|D\vw_1|^{p_1})^{1/p_1}_{\Omega_{2r}(x_0)}
\le N\sum_i(|\vf_i|^{p_1})^{1/p_1}_{\Omega_{2r}(x_0)}.
$$
Again, the existence and estimate for $w_1^\alpha$ are guaranteed by Theorem \ref{thm-200614-1046}. Then $v_1^\alpha:=u_1^\alpha-w_1^\alpha$ satisfies the conditions in Lemma \ref{lemma1}. Hence,
\begin{equation}\label{eqn-200616-1248}
\begin{split}
&(|\vv_1|^{p_3})^{1/p_3}_{\Omega_r(x_0)}
+(|D\vv_1|^{p_3})^{1/{p_3}}_{\Omega_r(x_0)}
\le N(|\vv_1|)_{\Omega_{2r}(x_0)}
+N(|D\vv_1|)_{\Omega_{2r}(x_0)}\\
&\le N(|\vu_1|)_{\Omega_{2r}(x_0)}+N(|D\vu_1|)_{\Omega_{2r}(x_0)} + N(|\vw_1|^{p_1})^{1/p_1}_{\Omega_{2r}(x_0)}+N(|D\vw_1|^{p_1})^{1/p_1}_{\Omega_{2r}(x_0)}\\
&\le N(|\vu_1|)_{\Omega_{2r}(x_0)}+N(|D\vu_1|)_{\Omega_{2r}(x_0)}
+N\sum_i(|\vf_i|^{p_1})^{1/p_1}_{\Omega_{2r}(x_0)},
\end{split}
\end{equation}
for arbitrarily large $p_3<\infty$. There is a similar decomposition in the interior of the domain. Now, let
\begin{equation*}
U := \sum_\alpha|u^{\alpha}_1| + \sum_{i,\alpha}|D_iu^\alpha_1|,\quad |F| := \sum_{i,\alpha}|f_i^{\alpha}|,
\end{equation*}
\begin{equation*}
V := \sum_\alpha|v^{\alpha}_1| + \sum_{i,\alpha}|D_iv^\alpha_1|,\quad \text{and}\quad W := \sum_\alpha|w^{\alpha}_1| + \sum_{i,\alpha}|D_iw^\alpha_1|.
\end{equation*}
We apply Theorem \ref{thm-200604-0840} with $p_1=p_2=\hat{p}$, $p_0=p$, $\nu=0$, and $p_3=ps/(s-1) + \epsi$, where $\epsi$ can be any positive number, and $\hat p$ and $s$ are from Lemma \ref{lem-200615-1006}. We then obtain
	\begin{equation*}
	\begin{split}
	\norm{\vu_1}_{W^1_{p,\omega}(\Omega)}&\le N \sum_i\norm{\vf_i}_{L_{p,\omega}(\Omega)} + N\norm{\vu_1}_{W^1_{p_1}(\Omega)}\\
	&\le N \sum_i\norm{\vf_i}_{L_{p,\omega}(\Omega)} + N\sum_i\norm{\vf_i}_{L_{p_1}(\Omega)}\\
	&\le N \sum_i\norm{\vf_i}_{L_{p,\omega}(\Omega)},
	\end{split}
\end{equation*}
where we used Theorem \ref{thm-200614-1046} in the second inequality.

The estimate for $u_2^\alpha$ can be obtained similarly, by choosing a different group of parameters. As before, we first solve
\begin{equation*}
\begin{cases}
D_i(a_{ij}^{\alpha\beta}D_j w_2^\beta)= g^\alpha\zeta\quad&\text{in}\quad \Omega,\\
a_{ij}^{\alpha\beta}D_j w_2^\beta n_i + \gamma^{\alpha\beta} w_2^\beta=0\quad &\text{on}\quad \partial \Omega.
\end{cases}
\end{equation*}
According to Theorem \ref{thm-200614-1046}, such solution exists, satisfying
$$
(|\vw_2|^{p_1})^{1/p_1}_{\Omega_r(x_0)}+(|D\vw_2|^{p_1})^{1/p_1}_{\Omega_r(x_0)}
\le Nr^{d/p_1^* - d/p_1}(|\vg|^{p_1^{*}})^{1/p_1^{*}}_{\Omega_{2r}(x_0)}.
$$
Now $v_2^\alpha=u_2^\alpha-w_2^\alpha$ satisfies the same equations as $v_1^\alpha$ in $\Omega_{2r}$. Hence, similar to \eqref{eqn-200616-1248}, the following holds for arbitrarily large $p_3<\infty$
\begin{equation*}
	\begin{split}
	&(|\vv_2|^{p_3})^{1/p_3}_{\Omega_r(x_0)}
+(|D\vv_2|^{p_3})^{1/{p_3}}_{\Omega_r(x_0)}
	\le N(|\vv_2|)_{\Omega_{2r}(x_0)}
+N(|D\vv_2|)_{\Omega_{2r}(x_0)}\\
	&\le N(|\vu_2|)_{\Omega_{2r}(x_0)}+N(|D\vu_2|)_{\Omega_{2r}(x_0)} + N(|\vw_2|^{p_1})^{1/p_1}_{\Omega_{2r}(x_0)}
+N(|D\vw_2|^{p_1})^{1/p_1}_{\Omega_{2r}(x_0)}\\
	&\le N(|\vu_2|)_{\Omega_{2r}(x_0)}+N(|D\vu_2|)_{\Omega_{2r}(x_0)}
	+Nr^{d/p_1^* - d/p_1}(|\vg|^{p_1^{*}})^{1/p_1^{*}}_{\Omega_{2r}(x_0)}.
	\end{split}
	\end{equation*}
	Define $U$, $V$, and $W$ as before, and
$$
F := \sum_\alpha|g^\alpha|.
$$
Now, applying Theorem \ref{thm-200604-0840} with
$$
p_1 = \hat{p},\quad p_2 = (\hat{p})^*,\quad p_3 = \frac{ps}{s-1} + \epsi,\quad p_0 = p_\omega,\quad\text{and}\quad  \nu = 1-(\hat{p})^*/\hat{p},
$$
we obtain
\begin{align*}
\norm{\vu_2}_{W^1_{p,\omega}}
&\le N \norm{\vu_2}_{W^1_{\hat{p}}(\Omega)}
+ N\norm{\vg}_{L_{p_\omega,\omega^{p_\omega/p}}(\Omega)}\\
&\le N \norm{\vg}_{L_{(\hat{p})^*}(\Omega)} + N \norm{\vg}_{L_{p_\omega,\omega^{p_\omega/p}}(\Omega)},
\end{align*}
where the second inequality is obtained by the unweighted estimate in Theorem \ref{thm-200614-1046}. Now, applying the properties of $A_p$ weights in Lemma \ref{lem-200615-1006}, we have $\omega\in A_{p/\hat{p}}$, and hence,
\begin{equation*}
\omega^{p_\omega/p} \in A_{p_\omega/\hat{p} + 1 - p_\omega/p}.
\end{equation*}
From the embedding in Lemma \ref{lem-200624-0616}, we have
\begin{equation*}
	\norm{\vg}_{L_{(\hat{p})^*}(\Omega)} \le N \norm{\vg}_{L_{p_\omega,\omega^{p_\omega/p}}(\Omega)}
\end{equation*}
because by the definition of $p_\omega$, we have
$$
\frac{p_\omega}{p_\omega/\hat{p} + 1 - p_\omega/p}=(\hat{p})^*\ge 1.
$$
The case (a) of the theorem is proved.

\textbf{Proof of Case (b)}. Note that the only difference in this case is the treatment of the $g^\alpha$ part. More precisely, here we only need to deal with \eqref{eqn-200604-0843-2} with $g^\alpha \in L_{q,\omega^a}$, where $q=dp/(d+p-1)$ and $a=(d-1)/(d+p-1)$. The proof is by duality. By Lemma \ref{lem-200615-1006}, $\omega_1:=\omega^{-p'/p}\in A_{p'}$. Using the result in (a) and Remark \ref{rem2.6} (b), for any function $h_i^\alpha,\tilde g^\alpha\in L_{p',\omega_1}$, there exists a unique solution $v_1^\alpha$ to the adjoint equation of \eqref{eqn-200622-1036}-\eqref{eqn-200527-1002} with $h_i^\alpha$ and $\tilde g^\alpha$ in place of $f_i^\alpha$ and $g^\alpha$. Furthermore, such solution satisfies
\begin{equation}\label{eqn-200624-0627-2}
	\sum_\alpha\norm{v_1^\alpha}_{W^1_{p',\omega_1}}\leq N \sum_{i,\alpha}\norm{h_i^\alpha} _{L_{p',\omega_1}}
+N\sum_{\alpha}\norm{\tilde g^\alpha} _{L_{p',\omega_1}}.
\end{equation}
In the following, the $\sup$ is always taken in the set
$$
\big\{\norm{h_i^\alpha}_{L_{p',\omega_1}}=\norm{\tilde g^\alpha}_{L_{p',\omega_1}}=1\big\},
$$
which is omitted in the computation.
Now, for any weak solution $\vu^2:=(u_2^\alpha)_\alpha$ to \eqref{eqn-200604-0843-2}, we have
\begin{align}
\norm{\vu_2}_{L_{p,\omega}} +
\norm{D\vu_2}_{L_{p,\omega}} =&\sup_{\norm{h_i^\alpha}_{L_{p',\omega_1}}=\norm{\tilde g^\alpha}_{L_{p',\omega_1}}=1}\left|\int_\Omega h_i^\alpha D_i u_2^\alpha-\tilde g^\alpha u_2^\alpha\right| \nonumber\\
		=& \sup_{h_i^\alpha,\tilde{g}^\alpha}\left|\int_\Omega a_{ij}^{\alpha\beta}D_j v_1^\beta D_i u_2^\alpha + \int_{\p\Omega}\gamma^{\alpha\beta}v_1^\beta u_2^\alpha \right| \nonumber\\
		=& \sup_{h_i^\alpha,\tilde{g}^\alpha}\left|\int_\Omega \vec{v}_1 \cdot \vec{g}\right| = \sup_{h_i^\alpha,\tilde{g}^\alpha}\left|\int_\Omega \vec{v}_1 \omega^{-a/q} \cdot\vec{g} \omega^{a/q}\right|\nonumber\\
		\leq& \sup_{h_i^\alpha,\tilde{g}^\alpha} N \norm{\vec{g}}_{L_{q,\omega^a}} \norm{\vec{v}_1}_{L_{q',\omega^{-aq'/q}}} \nonumber\\
		\leq& \sup_{h_i^\alpha,\tilde{g}^\alpha} N \norm{\vec{g}}_{L_{q,\omega^a}}\norm{\vec{v}_1}_{W^1_{p',\omega_1}}\label{eqn-200624-0619}\\
		\leq& \sup_{h_i^\alpha,\tilde{g}^\alpha} N \norm{\vec{g}}_{L_{q,\omega^a}}\Big(\norm{h_i^\alpha} _{L_{p',\omega_1}}+\norm{\tilde g^\alpha} _{L_{p',\omega_1}}\Big)\label{eqn-200624-0627}\\
		=& N \norm{\vec{g}}_{L_{q,\omega^a}}\notag.
\end{align}
Here, in \eqref{eqn-200624-0619} we used the embedding in Lemma \ref{lem-200624-0616} and H\"older's inequality, noting that
$$
q'=\frac{dp}{(d-1)(p-1)}=
\frac{dp'}{d-1}\quad\text{and}\quad \frac{aq'}{q} = \frac{p'}{p}.
$$
The estimate in \eqref{eqn-200624-0627} follows from \eqref{eqn-200624-0627-2}. This gives us the a priori estimate for the case (b). The solvability can be obtained by approximation. By the first embedding in Lemma \ref{lem-200624-0616} (a), we can choose a sufficiently large  $\bar{q}$ such that
\begin{equation}
                            \label{eq1.52}
L_{\bar q}(\Omega)\subset L_{p,\omega}(\Omega)\cap L_{q,\omega^a}(\Omega).
\end{equation}
Now we take a sequence of functions $\{f_{k,i}^\alpha\}_{k=1}^\infty \subset L_{\bar{q}}(\Omega)$ and $\{g_k^\alpha\}_{k=1}^\infty\subset L_{\bar{q}}$ such that
as $k\to \infty$,
\begin{equation}
                \label{eq1.54}
f_{k,i}^\alpha\to f_i^\alpha\quad\text{in}\,\,L_{p,\omega}(\Omega),
\quad g_k^\alpha\to g^\alpha\quad\text{in}\,\,L_{q,\omega^a}(\Omega).
\end{equation}
By the unweighted $W^1_{\bar{q}}$ solvability, we can find the corresponding solutions $\vu_k$. Using \eqref{eq1.52}, the weighted a priori estimate proved above, and \eqref{eq1.54}, we know that $\{\vu_k\}$ is a Cauchy sequence in $W^1_{p,\omega}$. Upon passing to the limit in the weak formulation of the equation, we see that the limiting function $\vu\in W^1_{p,\omega}$ is the solution of the original problem.  Hence, the theorem is proved.

We would like to mention that, here we used the $W^1_p$ estimate for conormal problems with two exponents $p=p_1$ and $p=p_3$ which depend on the properties of the weight $\omega$. The small parameter $\theta_0$ in the assumptions should be chosen accordingly.
%
%
%
%
%

\section{Scalar equations}
                                \label{sec5}
In this section, we prove Theorems \ref{thm-200617-1154}, \ref{thm-200618-1200}, and \ref{thm-200622-1138} for scalar equations. Here the only difference lies in the counterpart of Proposition \ref{prop-200611-0515} which gives a decomposition of solutions at small scales. In this section, we will prove such decomposition in Proposition \ref{prop-200625-0705}. Once we have Proposition \ref{prop-200625-0705}, all the three theorems can be proved in the same way as in the system case: actually the condition that $\p\Omega$ is locally close to a hyperplane/convex domain is only used in constructing the decomposition. For all the rest steps, we only require the boundary to be Lipschitz.

More precisely, compared to Proposition \ref{prop-200611-0515} for the system case, here in Proposition \ref{prop-200625-0705}, the estimate for the regular part $V$ is different. In the system case, the regular part satisfies the homogeneous conormal boundary condition on a hyperplane, which led to a Lipschitz estimate. While for the scalar case, it satisfies the homogeneous conormal boundary condition on part of the boundary of a convex domain. A global Lipschitz estimate for the Poisson equation was first established in \cite{MR2576900}. Here, we refer to a local version in \cite{MR3168044}, which was proved for the more general $p$-Laplacian equations.

\begin{proposition}\label{prop-200625-0705}
Suppose that $\Omega$ is a bounded domain, $q\in (1,\infty)$, and Assumptions \ref{ass-200617-1149} $(\theta)$ and \ref{ass-200519-1052} $(\theta)$ are satisfied with some $\theta\in (0,1)$. Furthermore, let $u\in \dot{L}^1_q(\Omega)$ be a weak solution to \eqref{eq11.20}-\eqref{eqn-200617-1155} with $f_i\in L_q(\Omega)$ and $g = \varphi = 0$. Then for any $\tilde{p}\in(1,q)$, $r<R_0$, and $x\in\overline{\Omega}$, we can find non-negative functions $V, W\in L_{\tilde{p}}(\Omega_{r/4\sqrt{d}M}(x))$ satisfying $|Du|\leq V + W$,
	\begin{equation}\label{eqn-200625-0706}
	(W^{\tilde{p}})^{1/\tilde{p}}_{\Omega_{r/(4\sqrt{d}M)}(x)} \le N \theta^{1/\tilde{p} - 1/q}(|Du|^{q})^{1/q}_{\Omega_r(x)} + N (F^{\tilde{p}})^{1/\tilde{p}}_{\Omega_{r}(x)},
	\end{equation}
	and
	\begin{equation}\label{eqn-200625-0707}
	\norm{V}_{L_\infty(\Omega_{r/(8\sqrt{d}M)}(x))} \le N (|Du|^{q})^{1/q}_{\Omega_r(x)} + N(F^{\tilde{p}})^{1/\tilde{p}}_{\Omega_{r}(x)},
	\end{equation}
where $F=\sum_i|f_i|$ and $N=N(d,\kappa,M,q,\tilde p)>0$.
\end{proposition}
\begin{proof}
As before, we only prove the case when $x\in\p\Omega$, where the coordinate system is taken as in Assumption \ref{ass-200519-1052}. We denote
\begin{equation*}
	c=\frac{1}{\sqrt{d}M}\quad\text{and}\quad  \Gamma_{cr}:=\p\Omega_{cr}^{+}\cap\{y^d=\psi(y')\}.
\end{equation*}
The proof follows similar steps with that of Proposition \ref{prop-200611-0515}. Indeed the ``reflection'' argument there can be adapted to this case.  We define the reflection with respect to the surface $\{x',\psi(x')\}$ by
$$
T(x)=(x',2\psi(x')-x_d).
$$
Clearly, we have $T^2=I$ and $\det(DT)=1$.
Then a direct calculation reveals that if $u$ is a weak solution to \eqref{eq11.20}-\eqref{eqn-200617-1155}, then it satisfies
\begin{equation*}
D_i(a_{ij}D_j u) = D_i \tilde f_i
\end{equation*}
in $\widetilde\Omega_r$ with the conormal boundary condition on $\{x:\ x'\in B_r',\ x_d=\psi(x')\}$,
where
\begin{align*}
	\tilde f_i&=f_i+1_{T(x)\in \Omega^-_{cr}}
	\big(f_i-a_{ij}D_j u \big)\circ T(x),
	\quad i=1,\ldots,d-1,\\
	\tilde f_d&=f_d-1_{\Omega^-_{cr}}\big(f_d-a_{dj}D_ju\big)\circ T(x) \\
	&\quad +1_{T(x)\in \Omega^-_{cr}}\sum_{j=1}^{d-1}2D_i\psi(x')\big(f_i-a_{ij}D_j u\big)\circ T(x).
\end{align*}
Note that for this step, we only require $\psi$ to be Lipschitz.

Now we construct the decomposition explicitly. We freeze the coefficients
$\bar{a}_{ij} := (a_{ij})_{B_r}$.
Let $w$ be the weak solution to
\begin{equation*}
	D_i(\bar{a}_{ij}D_j w) = D_i \tilde f_i + D_i((\bar{a}_{ij}-a_{ij})D_j u)\quad \text{in}\,\,\Omega_{cr}^+
\end{equation*}
with the conormal boundary condition on $\p\Omega_{cr^{+}}$. Since $\Omega_{cr}^+$ is convex, according to \cite{GS10}, such $w$ exists and satisfies
\begin{equation*}
	\norm{Dw}_{L_{\tilde{p}}(\Omega_{cr}^+)} \leq N(d,p,M) \norm{\tilde f_i + (\bar{a}_{ij}-a_{ij})D_j u}_{L_{\tilde{p}}(\Omega_{cr}^+)}.
\end{equation*}
Now define $W := |Dw|\mathbb{I}_{\Omega_{cr}^+} + |Du|\mathbb{I}_{\Omega_{cr}^-}$. Combining this and H\"older's inequality (cf. \eqref{eqn-200611-0448} and \eqref{eqn-200715-1236}), we have
\begin{equation*}
	(W^{\tilde{p}})^{1/\tilde{p}}_{\Omega_{cr}} \leq N (|f_i|^{\tilde{p}})^{1/\tilde{p}}_{\Omega_r} + N\theta^{1/\tilde{p}-1/q}(|Du|^q)^{1/q}_{\Omega_r}.
\end{equation*}
Now $v:=u-w$ satisfies
\begin{equation*}
	D_i(\bar{a}_{ij}D_j v) = 0 \quad \text{in}\,\,\Omega_{cr}^+
\end{equation*}
with the conormal boundary condition on $\Gamma_{cr}$. Noting that since $\bar{a}_{ij}$ is symmetric, by a linear transformation, $v$ is harmonic satisfying the zero Neumann boundary condition on the curved boundary which is still convex. Now we denote $V:= Dv\mathbb{I}_{\Omega_{cr}^{+}}$. By \cite[Theorem~1.1]{MR3168044} with $p=2$ (i.e., the linear case) together with a rescaling, covering, and iteration argument, we have
\begin{equation*}
\norm{V}_{L_\infty(\Omega_{cr/2}(x))} \leq N (V^{\tilde{p}})^{1/\tilde{p}}_{\Omega_{cr}(x)}.
\end{equation*}
Using \eqref{eqn-200625-0706} and the definition of $V$, we obtain
\begin{equation*}
	\begin{split}
	\norm{V}_{L_\infty(\Omega_{cr/2}(x))} 
	&\leq N (|Du|^{\tilde{p}})^{1/\tilde{p}}_{\Omega_{cr}(x)} + N(|Dw|^{\tilde{p}})^{1/\tilde{p}}_{\Omega_{cr}(x)}\\
	&\leq N (|Du|^{q})^{1/q}_{\Omega_r(x)} + N(F^{\tilde{p}})^{1/\tilde{p}}_{\Omega_{r}(x)},
	\end{split}
\end{equation*}
which gives \eqref{eqn-200625-0707}.
The proposition is proved.
\end{proof}

\begin{remark}
From the proof above, it is easily seen that instead of assuming $(a_{ij})$ to be symmetric, we only need to assume that $(a_{ij})$ is locally close to some constant symmetric matrix in the sense of the $L_1$ average.
\end{remark}

\end{document}